\documentclass[twoside, 12pt, a4paper]{amsart}
\usepackage[a4paper,hmargin=1.25in,vmargin=1.3in]{geometry}
\usepackage{amscd}
\usepackage{verbatim}
\usepackage{color}
\usepackage{comment}
\usepackage{array}
\usepackage{multirow}

\input xy
\xyoption{all}

\usepackage{amssymb}
\usepackage{hyperref}
\usepackage{tikz-cd}
\usepackage{tikz}

\DeclareMathAlphabet{\cat}{OT1}{cmss}{m}{sl}

\newtheorem*{theorem*}{Theorem}
\newtheorem{theorem}{Theorem}[section]

\newtheorem{proposition}[theorem]{Proposition}
\newtheorem{lemma}[theorem]{Lemma}
\newtheorem{corollary}[theorem]{Corollary}

\theoremstyle{definition}
\newtheorem{remark}[theorem]{Remark}

\newtheorem{example}[theorem]{Example}

\newtheorem{dfn}[theorem]{Definition}

\newcommand{\tens}{\otimes}

\newcommand{\CH}{\operatorname{CH}}

\renewcommand{\Im}{\operatorname{Im}}

\newcommand{\ind}{\operatorname{\hspace{0.3mm}ind}}

\newcommand{\Spec}{\operatorname{Spec}}
\newcommand{\Spin}{\operatorname{Spin}}
\newcommand{\HSpin}{\operatorname{HSpin}}
\newcommand{\PGO}{\operatorname{PGO}}

\newcommand{\SL}{\operatorname{SL}}

\newcommand{\GL}{\operatorname{GL}}

\newcommand{\SO}{\operatorname{SO}}

\newcommand{\Z}{\mathbb{Z}}
\newcommand{\N}{\mathbb{N}}

\title[Upper Bounds on the Torsion Index of Half-Spin Groups] 
{Upper Bounds on the Torsion Index of Half-Spin Groups}

\author
[S.~Baek, R.~Devyatov] {Sanghoon Baek, Rostislav Devyatov}

\address[Sanghoon Baek]{Department of Mathematical Sciences, 
	KAIST,
	291 Daehak-ro, Yuseong-gu,
	Daejeon 305-701,
	Republic of Korea}
\email{sanghoonbaek@kaist.ac.kr}
\urladdr{https://mathsci.kaist.ac.kr/~sbaek/}

\address[Rostislav Devyatov]{Laboratory of Algebraic Geometry and its Applications, 
	Department of Mathematics,
	National Research University Higher School of Economics,
	6 Usacheva str.,
	Moscow 119048,
	Russian Federation}
\email{deviatov@mccme.ru}
\urladdr{https://www.mccme.ru/~deviatov/}

\thanks{The work of both authors was supported by the Samsung Science and Technology Foundation under Project Number SSTF-BA1901-02. Rostislav Devyatov thanks KAIST for hospitality and support. He was also partially supported by the HSE University Basic Research Program and expresses his gratitude to the Laboratory of Algebraic Geometry and its Applications at HSE University for their hospitality.}

\date
{\today}

\begin{document}
	
\begin{abstract}
The torsion index of split simple groups has been extensively studied, notably by Totaro, who calculated the torsion indexes of the spin groups and $E_{8}$ in \cite{Totaro_E8} and \cite{Totaro_Spin}, respectively. The aim of this paper is to provide upper bounds for the torsion index of half-spin groups, the only remaining case in the calculation of torsion indexes for split simple groups. We present general upper bounds for the torsion index of half-spin groups, showing that, except for certain exceptional cases, it is at most twice that of the corresponding spin groups. For these exceptional cases, the torsion index is bounded above by at most $2^3$ times that of the spin groups. Our results also reveal that in many cases, the torsion index of half-spin groups coincides with that of the spin groups.
\end{abstract}

\maketitle
\tableofcontents

\section{Introduction}\label{secone}
Consider a variety $X$ over a field $k$. The \emph{index} of $X$, denoted by $\ind(X)$, is defined as
\begin{equation*}
	\ind(X)=\gcd \{\,[K:k]\,\,|\,\,K/k \text{ is a finite extension such that } X(K)\neq \emptyset \},
\end{equation*}
or equivalently, as the gcd of the degrees of all closed points of $X$.

Now let $G$ be an algebraic group over a field $k$. The \emph{torsion index} of $G$, denoted by $\tau(G)$, is defined by
\begin{equation}\label{torsionind}
	\tau(G)=\operatorname{lcm}\{\, \ind(E)\,\,|\,\, E \text{ is a } G\text{-torsor over a field extension of } k\}.
\end{equation}
The prime divisors of the torsion index $\tau(G)$ are called the \emph{torsion primes} of $G$. When $G$ is of type $B$ or $D$, the only torsion prime is $2$. We denote the $2$-adic valuation of $\tau(G)$ by $\tau_{2}(G)$. In particular, if $G$ is a split semisimple group, then $\ind(E)=\ind(E/B)$, where $B$ is a Borel subgroup of $G$. Thus, one can replace $\ind(E)$ by $\ind(E/B)$ in (\ref{torsionind}). Indeed, one can show that
if  $E$ is a {\em generic $G$-torsor}, i.e., the generic fiber of the quotient map $\GL(n)\to \GL(n)/G$ induced by an embedding $G\hookrightarrow \GL(n)$ for some $n\geq 1$, then 
\begin{equation*}
	\tau(G)=\ind(E/B).
\end{equation*}

Let $G$ be a split semisimple group over a field $k$, and let $T$ be a split maximal torus of $G$. The symmetric algebra $S(T^{*})$ of the character group $T^{*}$ is isomorphic to the Chow ring of the classifying space of $T$,
\begin{equation*}
	S(T^{*})\simeq \CH(BT),	
\end{equation*}
where the isomorphism is given by mapping $\chi\in T^{*}$ to $c_{1}(\chi)$.

Let $P$ be a special parabolic subgroup of $G$ containing $T$, i.e., $H^{1}(K,P)=\{1\}$ for any field extension $K/k$, and let $W_{P}$ be the Weyl group of $P$. Then, the canonical morphism 
\begin{equation}\label{eq:chbpchbtwp}
	\CH(BP)\to \CH(BT)^{W_{P}}=S(T^{*})^{W_{P}}
\end{equation}
is an isomorphism. Hence, the morphism $\CH(BP)\to \CH(G/P)$ given by the pullback of the structure morphism $G\to \Spec(k)$ induces the following morphism:
\begin{equation}\label{eq:variphi}
	\varphi: S(T^{*})^{W_{P}}\to \CH(G/P).
\end{equation}
and the torsion index $\tau(G)$ of $G$ is given as follows:
\begin{equation*}
	\tau(G)=[\CH_{0}(G/P):\Im \varphi \cap \CH_{0}(G/P)].
\end{equation*}

In \cite{Totaro_Spin}, Totaro computed the torsion indices of the even split spin groups $\Spin(2n)$. To formulate this result, first, for any nonnegative integer $s$, define
\begin{equation}\label{eq:defnzero}
	n_{0}:=\text{smallest integer } n \text{ such that } {n\choose 2}+1\geq 2^{2s}.
\end{equation}
Second, recall that it is known that $\tau_{2}(\Spin(2n))=0$ for $1\leq n\leq 3$ and $\tau_{2}(\Spin(8))=1$.

From this, the torsion indices of the even split spin groups can be calculated inductively using the following definition and theorem. For any integer $s \geq 2$, define
\begin{equation}\label{eq:defmzero}
m_{0}:=\text{smallest ingeter } m\geq 0 \text{ such that } 2m-\tau_{2}(\Spin(2m+2))>s-3.
\end{equation}
(Note that the above values of $\tau_{2}(\Spin(2n))$ for $n \leq 4$ make it possible to compute the values of $m_0$ for $s \leq 7$: e.g., for $s=7$, we get $m_0=3$.)

\begin{theorem}\cite[Theorem 0.1]{Totaro_Spin}\label{thm:totaromain}
	Let $n\in (2^{s},2^{s+1}]$ for some integer $s\geq 2$. Let $n_{0}$ and $m_{0}$ denote the integers in $(\ref{eq:defnzero})$ and $(\ref{eq:defmzero})$. Then,
	$$
	\tau_{2}(\Spin (2n))=
	\begin{cases}
		n-2s+1 & \text{ if } n\in (2^{s},\,2^{s}+m_{0}], \\
		n-2s & \text{ if } n\in (2^{s}+m_{0}, n_{0}),\\
		n-2s-1 & \text{ if } n\in [n_{0}, 2^{s+1}]  .
	\end{cases}
	$$
\end{theorem}

Here and further, throughout the whole paper, we assume that $(a, b]$ denotes the set of integers $n$ such that $a < n \leq b$, $[a, b]$ denotes the set of integers $n$ such that $a \leq n \leq b$, and $(a, b)$ denotes the set of integers $n$ such that $a < n < b$.

Observe that this theorem, for values of $s$ in $[2, s_0]$ (where $s_0$ is a fixed integer), allows one to compute $m_0$ for values of $s$ up to $2^{s_0 + 1} + 2s_0 + 1$, which is greater than $s_0$, so the induction can indeed proceed. In fact, $m_0$ is approximately given by $s - 3 - 2\log_2 s$, and one can verify that $2^s + m_0 \in [2^s, 2^s + 2^{s-3}]$. For $n_0$, we have $n_0 \in [2^s + 2^{s-2}, 2^s + 2^{s-1}]$, and approximately, $n_0$ is close to $\sqrt{2} \cdot 2^s$.

For $2\leq s\leq 10$, we provide the following table for the values of $n_{0}$ and $m_{0}$:

\bigskip

\begin{center}
	\begin{tabular}{|c|c|c|c|c|c|c|c|c|c|c|}
		\hline
		$s$ & \rule{0pt}{2.5ex} $2$ & $3$ & $4$ & $5$ & $6$ & $7$ & $8$ & $9$ & $10$ & $\cdots$  \\ \hline
		$n_{0}$ & \rule{0pt}{2.5ex} $6$ & $12$ & $24$ & $46$ & $91$ & $182$ & $363$ & $725$ & $1449$  & $\cdots$ \\ \hline
		$m_{0}$ & \rule{0pt}{2.5ex} $0$ & $1$ & $1$ & $2$ & $2$ & $3$ & $4$ & $4$ & $5$ & $\cdots$ \\ \hline
	\end{tabular}
\end{center}

\bigskip

Now we consider the half-spin groups $\HSpin(2n)$ (also called semi-spin groups) with even integer $n\geq 2$. The torsion index of $\HSpin(2n)$ for small $n$ is known or will be estimated below as follows:

For $n=2$, we have $\Spin(4)\simeq \Spin(3)\times \Spin(3)$, thus $\HSpin(4)\simeq \Spin(3)\times \SO(3)$. Since $\tau(\SO(3))=2$ and $\tau(\Spin(3))=1$, we get 
\begin{equation*}
	\tau(\HSpin(4))=2.
\end{equation*}

For $n=4$, we will check that $\tau(\HSpin(8))\leq 2^{3}$ (See Corollary \ref{cor:2powerthreeupp} or Lemma \ref{lem:s22mchoose2}). In fact, by triality, $\HSpin(8)\simeq \SO(8)$, so
\begin{equation*}
	\tau(\HSpin(8))=2^{3}.
\end{equation*}

For even $n\geq 6$, we will get an elementary upper bound for $\tau(\HSpin(2n))$ in Lemma \ref{lem:s22mchoose2}:
	\begin{equation*}
	\tau_{2}(\HSpin(2n))\leq \begin{cases} n-S_{2}\big({n\choose 2}\big) & \text{ if } v_{2}(n)=1,\\
		n-S_{2}\big({n\choose 2}\big)+1 & \text{ if } v_{2}(n)\geq 2,\\
	\end{cases}
\end{equation*}
where $S_{2}({n\choose 2})$ denotes the sum of the base-$2$ digits of ${n\choose 2}$. In general, this elementary bound is less sharp than the bound we will provide in Theorem \ref{thm:main}. However, for $n= 6$ and $n = 8$, both bounds coincide.

Moreover, the upper bounds with $n=6$ and $n=8$ recover \cite[Theorem 5.1]{Totaro_E8} and the proof of \cite[Lemma 3.1]{Totaro_Splitting}, respectively, That is, $\tau(\HSpin(12))\leq 2^{2}$ and $\tau(\HSpin(16))\leq 2^{6}$. In fact, as noted by Totaro in \cite[p. 225, p. 238]{Totaro_E8}, 
\begin{equation*}
	\tau(\HSpin(12))=2^{2} \text{ and } \tau(\HSpin(16))=2^{6}.
\end{equation*}
Though no proof was provided there, we establish these equalities in Example \ref{ex:hspintwelve} and Corollary \ref{cor:lowerbdneight}, respectively.

For half-spin groups with arbitrary even $n$, we establish the following upper bounds for the torsion index in this paper. Namely, unless $n$ is a power of $2$ or of the form $3 \cdot 2^k$ for some $k$, the torsion index is at most twice that of the corresponding spin group. For these exceptional values of $n$, the torsion index of the half-spin group is bounded above by at most $2^3$ times that of the corresponding spin group.

Moreover, the obvious but notable lower bound is always given by the torsion index of the spin group (see Lemma \ref{lem:basicbound}). In many cases, the torsion index of the half-spin groups coincides with that of the spin groups.

\begin{theorem}\label{thm:main}
	Let $n\geq 4$ be an even integer and let $s$ be the integer such that $n \in (2^s, 2^{s+1}]$. Then, 
	\begin{equation*}
		\tau(\HSpin (2n))\leq \begin{cases} 2\cdot \tau(\Spin (2n)) & \text{ if } n\in (2^{s}, 2^{s+1})\setminus \{3\cdot 2^{s-1}\} \text{ or } n=6,\\
		2^{2}\cdot \tau(\Spin (2n)) & \text{ if } n=4 \text{ or } n=12,\\
		2^{3}\cdot\tau(\Spin (2n)) & \text{ if } n=2^{s+1}, s \ge 2 \text{ or } n=3\cdot 2^{s-1}, s \ge 4.
	\end{cases}
	\end{equation*}
	Moreover, we have
	\begin{equation*}
		\tau(\HSpin (2n))=\tau(\Spin (2n))
	\end{equation*}
	if $n\in (2^{s}, 2^{s}+m_{0}]\cup (2^{s}+2^{s-3}, 2^{s}+2^{s-2})\cup (2^{s}+2^{s-2}, n_{0})\cup (2^{s}+2^{s-1}, 2^{s+1})$, or $v_{2}(n)=1$, $s\geq 3$, and $n\in (2^s + m_0 + 1, 2^{s}+2^{s-3})\cup (n_0, 2^s + 2^{s-1})$, where $n_{0}$ and $m_{0}$ are as defined in $(\ref{eq:defnzero})$ and $(\ref{eq:defmzero})$.
\end{theorem}
\begin{proof}
The statement follows from Corollaries \ref{cor:intervalupperbd}, \ref{cor:2powerthreeupp}, and \ref{cor:3powerthreeupp}.
\end{proof}

The intervals where $\tau(\HSpin (2n))=\tau(\Spin (2n))$ are highlighted in red color, specifically for $n \in (2^s, 2^s + m_0] \cup (2^s + 2^{s-3}, 2^s + 2^{s-2}) \cup (2^s + 2^{s-2}, n_0) \cup (2^s + 2^{s-1}, 2^{s+1})$, as shown below:

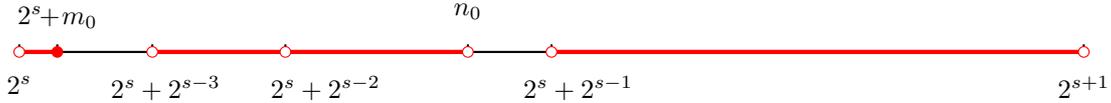
\begin{center}
	\begin{tikzpicture}
		\coordinate (A) at (0,0);                 
		\coordinate (B) at (0.5,0);                 
		\coordinate (C) at (1.75,0);                 
		\coordinate (D) at (3.5,0);                 
		\coordinate (E) at (5.9,0);                
		\coordinate (F) at (7,0);                
		\coordinate (G) at (14,0);                
		
		\draw[thick, -] (A) -- (G);
		
		\foreach \pt in {A,B,C,D,E,F,G} {
			\draw[thick] (\pt) -- ++(0,0.1);
		}
		
		\node[below=5pt, font=\footnotesize] at (A) {$2^s$};
		\node[above=5pt, font=\footnotesize] at (B) {$2^s\!\!+\!m_0$};
		\node[below=5pt, xshift=5pt, font=\footnotesize] at (C) {$2^s+2^{s-3}$};
		\node[below=5pt, xshift=15pt,font=\footnotesize] at (D) {$2^s+2^{s-2}$};
		\node[above=8.1pt,font=\footnotesize] at (E) {$n_0$};
		\node[below=5pt, xshift=10pt, font=\footnotesize] at (F) {$2^s+2^{s-1}$};
		\node[below=5pt,font=\footnotesize] at (G) {$2^{s+1}$};
		
		\draw[red, ultra thick] (A) -- (B);
		\draw[red, fill=white] (A) circle [radius=0.07];  
		\draw[red, fill=red] (B) circle [radius=0.07];    
		
				\draw[red, ultra thick] (C) -- (D);
			\draw[red, fill=white] (C) circle [radius=0.07];  

		\draw[red, ultra thick] (D) -- (E);
		\draw[red, fill=white] (D) circle [radius=0.07];  
		\draw[red, fill=white] (E) circle [radius=0.07];  
		
		\draw[red, ultra thick] (F) -- (G);
		\draw[red, fill=white] (F) circle [radius=0.07];  
		\draw[red, fill=white] (G) circle [radius=0.07];    
		
	\end{tikzpicture}
\end{center}

Note that the size of interval $[n_{0}, 2^{s}+2^{s-1}]$, where $n_{0}\approx \sqrt{2}\cdot 2^{s}\approx 1.414\cdot 2^{s}$, constitutes approximately $8.6\%$ of the length of the total interval $[2^{s},2^{s+1}]$. On the other hand, the size of the interval $[2^{s}, 2^{s}+m_{0}]$, where $m_{0}\approx s-3-2\log_{2}s$ is relatively small when $s$ is large. When we ignore the interval $[2^{s}, 2^{s}+m_{0}]$, the total length of the intervals, where $\tau(\HSpin (2n))=\tau(\Spin (2n))$, is approximately $50\%+(37.5\%-8.6\%)=78.9\%$ of the length of the total interval $[2^{s},2^{s+1}]$. And even outside these intervals, for most of the values of $n$ divisible by 2, but not by 4, we still have $\tau(\HSpin (2n))=\tau(\Spin (2n))$.

\section{Preliminaries}
In this section, specifically subsections \ref{subsec:two-step}, \ref{subsec2:two-step}, and \ref{subsec3:two-step}, we recall some basic results on the Chow ring $\CH(X)$ of a two-step flag variety $X$ of type $D_n$.

For even $n$, we revisit the Weyl group invariants in the Chow ring of the classifying space of split tori for all simple groups of type $D_n$ and discuss their images under the morphism (\ref{eq:variphi}) in $\CH(X)$.

In the final part of subsection \ref{subsec3:two-step}, we provide an elementary upper bound for the torsion index of the half-spin groups (see Lemma \ref{lem:s22mchoose2}) as discussed in the Introduction. 

For more details and general theory, we refer the reader to \cite{KM}.

In subsection \ref{subsec4:two-step}, we recall some useful relations for the images of (\ref{eq:variphi}) in $\CH(X)$, which were developed in the proof of \cite[Theorem 7.1]{Totaro_E8}. We also provide Lemma \ref{lem:expandtlarge2}, which offers a new relation for these images. For further details, see \cite{Totaro_E8}.

\subsection{The Chow ring of a two-step flag variety $X$ of type $D_n$}\label{subsec:two-step}

Let $q$ be a hyperbolic quadratic form on a $2n$-dimensional vector space over a field $k$ with a basis given by pairwise orthogonal hyperbolic pairs $\{v_{i}, u_{i}\}$ ($1\leq i\leq n$). Let $\tilde{G}=\Spin(q)$, i.e., $\tilde{G}=\Spin(2n)$.

Let $V_{1}$ denote the $1$-dimensional subspace spanned by $\{v_{1}\}$ and let $V_{n}$ denote the the $n$-dimensional totally isotropic subspace spanned by $\{v_{1},\ldots, v_{n}\}$. We write $\tilde{P}:=\tilde{P}_{1,n}$ for the stabilizer of the flag $V_{1}\subset V_{n}$ in $\tilde{G}$. Then, the semisimple part of $\tilde{P}$ is isomorphic to the special linear group $\SL(n-1)$.

Let $Y$ be the connected component of the scheme of maximal $n$-dimensional totally isotropic subspaces such that the point representing $V_n$ belongs to $Y$. Then, $Y$ is isomorphic to the variety of $(n-1)$-dimensional totally isotropic subspaces of the split $(2n-1)$-dimensional quadratic form over $k$. Hence, the Chow ring $\CH(Y)$ is isomorphic to the polynomial ring $\mathbb{Z}[e_{1},\ldots, e_{n-1}]$ in 
$n-1$ variables modulo the relations
\begin{equation}\label{chowringY}
	e_{i}^{2}-2e_{i-1}e_{i+1}+2e_{i-2}e_{i+2}-\cdots +(-1)^{i}e_{2i}=0
\end{equation}
for all $i\geq 1$, where we set $e_{i}=0$ for $i\geq n$ (see \cite[\S 86]{EKM} for the exact definition of $e_i$s and for further details).

Let $X=\tilde{G}/\tilde{P}$. Then, the natural projection $\pi:X\to Y$ is the projective bundle $\mathbb{P}(E)\to Y$ associated with the tautological $n$-dimensional vector bundle $E$ over $Y$. Therefore, 
\begin{equation}\label{chowringXprime}
	\CH(X)=\frac{\CH(Y)[t]}{(t^{n}+c_{1}t^{n-1}+\cdots +c_{n})},
\end{equation}
where $t=c_{1}(L)$ with $L=\mathcal{O}_{\mathbb{P}(E)}(1)$ and $c_{i}=c_{i}(E)$ for $1\leq i\leq n$. We have $c_n=0$ (see \cite[Theorem 2.1]{Kar})  and $c_i = (-1)^i 2e_i$ for $1 \le i \le n-1$ (see \cite[Proposition 86.13]{EKM}). We also set $c_{i}=0$ for $i\geq n+1$. 

Since $c_{n}=0$, the relation in (\ref{chowringXprime}) becomes
\begin{equation}\label{relationtandc}
	t^{n}+c_{1}t^{n-1}+\cdots +c_{n-1}t=0.
\end{equation}

Since $c_{i}=(-1)^{i}2e_{i}$ for $1\leq i\leq n-1$, it follows from (\ref{chowringY}) that the Chern classes $c_{1},\ldots, c_{n-1}$ satisfy the following relations:
\begin{equation}\label{eq:relationcitwo}
	c_{i}^{2}-2c_{i-1}c_{i+1}+2c_{i-2}c_{i+2}-\cdots +(-1)^{i}2c_{2i}=0
\end{equation}
for $1\leq i\leq n-1$ and the Chow ring $\CH(X)$ is isomorphic to
\begin{equation}\label{chowX}
		\frac{\mathbb{Z}[t,e_{1},\ldots, e_{n-1}]}{\big(t^{n}-2e_{1}t^{n-1}+\cdots +(-1)^{n-1}2e_{n-1}t,\,\, e_{i}^{2}-2e_{i-1}e_{i+1}+\cdots +(-1)^{i}e_{2i}\big)}
\end{equation}
for $1\leq i\leq n-1$. Observe that $\dim X=\frac{(n-1)(n+2)}{2}$ and 
\begin{equation}\label{chowringX}
	\CH_{0}(X)=\mathbb{Z}\cdot x_{0},
	\,\, \text{where } 
	x_{0}
	=t^{n-1}\prod_{i=1}^{n-1}e_{i}.
\end{equation}
We will also frequently use the relation that any monomial in $e_i$s of degree bigger than $\dim Y = n(n-1)/2$ is zero.
Also, it follows from \cite[Theorem 86.12]{EKM} and (\ref{chowringXprime}) that $\CH(X)$ (and hence its subrings) viewed as an abelian group is finitely generated and free.

\subsection{The Weyl group invariants in $\CH(BT)$}\label{subsec2:two-step}

Let $\tilde{T}$ be a maximal torus of $\tilde{G}=\Spin(2n)$. Then, the character lattice $\tilde{T}^{*}$ of $\tilde{T}$ is given by
\begin{equation*}
	\tilde{T}^{*}=\mathbb{Z}\cdot x_{1}+\cdots +\mathbb{Z}\cdot x_{n}+\mathbb{Z}\cdot y,
\end{equation*}
where $2y=x_{1}+\cdots +x_{n}$. Let $x_{i}'=x_{i}-x_{1}$ for $i=2,\ldots, n$ and let $s_{i}$ denote the elementary symmetric polynomials on $x_{i}'$. Then, the Weyl group $S_{n-1}$ of $\tilde{P}_{1,n}$, denoted by $W$, permutes $x_{i}'$. Moreover, the action of $W$ on both $y$ and $x_{1}$ is trivial. Hence,
\begin{equation*}
	\CH(B\tilde{T})^{W}=S(\tilde{T}^{*})^{W}=\Z[y,x_{1},s_{2},\ldots, s_{n-1}],
\end{equation*}
see \cite[\S 8.4]{KM}.

From now on, we shall assume that $n$ is even. Then, the center $Z(\tilde{G})$ of $\tilde{G}$ is isomorphic to $\mu_{2}\times \mu_{2}$. Let $G=\tilde{G}/\mu$, where $\mu$ denotes the subgroup $1\times \mu_{2}$ or $\mu_{2}\times 1$ of $Z(\tilde{G})$, i.e., $G=\HSpin(2n)$ and let $T$ be a maximal torus of $G$. Then,
\begin{equation}\label{eq:ChowBTSTstar}
	\CH(BT)^{W}=S(T^{*})^{W}=\Z[y, 2x_{1}, s_{2},\ldots, s_{n-1}],
\end{equation}
see \cite[\S 8.4]{KM}. Similarly, let $G'=\tilde{G}/\mu'$, where $\mu'\simeq \mu_{2}$ denotes the diagonal subgroup of $Z(\tilde{G})$, i.e., $G'=\SO(2n)$, and let $\bar{G}=\tilde{G}/Z(\tilde{G})$, i.e., $\bar{G}=\PGO^{+}(2n)$. For the maximal tori $T'$ and $\bar{T}$ of $G'$ and $\bar{G}$, respectively, we have
\begin{equation*}
	\CH(BT')^{W}=\Z[x_{1}, s_{1}, s_{2},\ldots, s_{n-1}]\,\, \text{ and } \CH(B\bar{T})^{W}=\Z[2x_{1}, s_{1},\ldots, s_{n-1}],
\end{equation*}
see \cite[\S 8.4]{KM}. Note that $\bar{T}^{*}\subset T^{*}\subset \tilde{T}^{*}$ and $S(\bar{T}^{*})^{W}\subset S(T^{*})^{W}\subset S(\tilde{T}^{*})^{W}$.

An elementary relation between the torsion indices of $\bar{G}$, $G$, and $\tilde{G}$ is as follows:

\begin{lemma}\label{lem:basicbound}
	We have $\tau(\tilde{G})\,\,|\,\, \tau(G)\,\,|\,\, \tau(\bar{G})$. In particular, $$\tau(\Spin(2n))\leq \tau(\HSpin(2n)).$$
\end{lemma}
\begin{proof}
		Let $\tilde{P}$ denote the parabolic subgroup $\tilde{P}_{1,n}$ of $\tilde{G}$, and let $W$ be the Weyl group of $\tilde{P}$. Since the semisimple part of $\tilde{P}$ is $\SL(n-1)$, the group $\tilde{P}$ is special. Define $P = \tilde{P}/\mu$ and $\bar{P} = \tilde{P}/Z(\tilde{G})$. Since the semisimple part $\SL(n-1)$ of $\tilde{P}$ intersects $Z(\tilde{G})$ trivially, the parabolic subgroups $P$ and $\bar{P}$ are also special. 
	
	Since $W = W_{P} = W_{\bar{P}}$, it follows from (\ref{eq:chbpchbtwp}) that
	\[
	\CH(B\tilde{P}) \simeq S(\tilde{T}^{*})^{W}, \quad \CH(BP) \simeq S(T^{*})^{W}, \quad \text{and} \quad \CH(B\bar{P}) \simeq S(\bar{T}^{*})^{W}.
	\]
	Moreover, we obtain the corresponding morphisms 
	\begin{equation}\label{eq:varphitildebar}
		\tilde{\varphi}:S(\tilde{T}^{*})^{W}\to \CH(\tilde{G}/\tilde{P}),\, \varphi:S(T^{*})^{W}\to \CH(G/P),\, \bar{\varphi}:S(\bar{T}^{*})^{W}\to \CH(\bar{G}/\bar{P})
	\end{equation}
	 as in (\ref{eq:variphi}). Since $\tilde{G}/\tilde{P} = G/P = \bar{G}/\bar{P}$ and $S(\bar{T}^{*})^{W}\subset S(T^{*})^{W}\subset S(\tilde{T}^{*})^{W}$, it follows that 
	 $\varphi = \tilde{\varphi} \vert_{S(T^{*})^{W}}$ and $\bar{\varphi} = \tilde{\varphi} \vert_{S(\bar{T}^{*})^{W}}$, 
	 proving the statement.\end{proof}

\subsection{Subrings of $\CH(X)$}\label{subsec3:two-step}

Let $\hat{R}=\CH(X)$. The morphisms $\varphi$ and $\tilde{\varphi}$ for $G$ and $\tilde{G}$ as in (\ref{eq:variphi}) and (\ref{eq:varphitildebar}) are given by
\begin{equation}\label{eq:imageofvarphi}
	\begin{split}
		\varphi: \CH(BT)^{W}=\Z[y, 2x_{1}, s_{2},\ldots, s_{n-1}]&\to  \hat{R}=\Z[t, e_{1},\ldots, e_{n-1}]/I \\
		\tilde{\varphi}: \CH(B\tilde{T})^{W}=\Z[y, x_{1}, s_{2},\ldots, s_{n-1}]&\to  \hat{R}=\Z[t, e_{1},\ldots, e_{n-1}]/I \\
		y&\mapsto  -e_{1}\\
		x_{1}&\mapsto  -t\\
		s_{i}&\mapsto  c_{i}\big(\frac{\pi^{*}E}{L^{*}}\tens L\big),\,
		i=2,\ldots, n-1,\end{split}
\end{equation}
where $I$ denotes the ideal in the denominator of (\ref{chowX}) and $L^{*}$ denotes the dual of $L$, i.e., $L^{*}=\mathcal{O}_{\mathbb{P}(E)}(-1)$ (to see this, similarly to \cite[\S 8.4]{KM}, we apply \cite[Lemma 7.1]{KM} to $E$, to $L^*$, and to $\frac{\pi^{*}E}{L^{*}}\tens L$, but compute the result in $\CH(X)$ rather than in $\CH(X)/2\CH(X)$). Note that
\begin{equation}\label{chernclassel}
	\begin{split}
		c_{i}\big(\frac{\pi^{*}E}{L^{*}}\tens L\big)=c_{i}(E\tens L)&=\sum_{k=0}^{i}{n+k-i\choose k}c_{1}(L)^{k}c_{i-k}(E)\\
		&={n\choose i}t^{i}+2\sum_{k=0}^{i-1}{n+k-i\choose k}(-1)^{i-k}t^{k}e_{i-k}
	\end{split}
\end{equation}
for any $1\leq i\leq n$. For $\varphi'$ and $\bar{\varphi}$, we also need $\varphi(s_1)$, but by \cite[Lemma 7.1]{KM}, we also get
$\varphi(s_1)=c_{1}(\frac{\pi^{*}E}{L^{*}}\tens L)=c_{1}(E \otimes L)$.
We simply write $d_{i}$ for the Chern class $c_{i}(E\tens L)$ in (\ref{chernclassel}), i.e., 
\begin{equation}\label{chernclassd}
	d_{i}={n\choose i}t^{i}+{n-1\choose i-1}t^{i-1}c_{1}+\cdots+{n-i+1\choose 1}tc_{i-1}+{n-i\choose 0}c_{i}
\end{equation}
for $1\leq i\leq n-1$.

We write $\langle t, e_{i}\rangle$ for the Chow ring $\hat{R}=\CH(X)$ in (\ref{chowX}) (i.e., $t$ and $e_{i}$ generate $\hat{R}$). By (\ref{eq:ChowBTSTstar}) and (\ref{eq:imageofvarphi}), we have a subring
\begin{equation*}
	\Im(\varphi)=\langle 2t, d_{i}, e_{1}\rangle
\end{equation*}
of $\hat{R}$. In addition, we have the following subrings of $\hat{R}$:
\begin{equation*}
	\Im(\tilde{\varphi})=\langle t, 2e_{i},e_{1}\rangle,\,\,\,  \Im(\bar{\varphi})=\langle 2t, d_{i}\rangle,\,\,\, \text{ and }\,\,\, \Im(\varphi')=\langle t, 2e_{i}\rangle,
\end{equation*}
where $\tilde{\varphi}$ and $\bar{\varphi}$ denote the morphisms in (\ref{eq:varphitildebar}), and $\varphi':\CH(BP')\to \CH(G'/P')=\CH(G/P)$ with $P'=\tilde{P}/\mu'$.

To simplify the notation, we set
\begin{equation*}
	R=\Im(\varphi),\,\, \tilde{R}=\Im(\tilde{\varphi}),\,\, \bar{R}=\Im({\bar{\varphi}}),\,\,  \text{ and }\,\, R'=\Im(\varphi').
\end{equation*}
Then, the following diagram exhibits the inclusions of the above subrings:

\begin{center}
	\begin{tikzcd}[column sep=0.6cm, row sep=1.3cm]
		{} & \hat{R}=\langle t, e_{i}\rangle\arrow[dash]{d} &\\
		{} & \tilde{R}=\langle t, 2e_{i}, e_{1}\rangle=\langle t, c_{i}, e_{1}\rangle\arrow[dash]{dr}\arrow[dash, ,labels=above left]{dl}
		& \\
		R=\langle 2t, d_{i}, e_{1}\rangle\arrow[dash,labels=below left]{dr} & & R'=\langle t, 2e_{i}\rangle=\langle t, c_{i}\rangle=\langle t, d_{i}\rangle\arrow[dash]{dl} \\
		& \bar{R}=\langle 2t, d_{i}\rangle
	\end{tikzcd}
\end{center}

Let $[A:B]_{0}$ denote the index of the degree $0$ part of $B$ in the degree $0$ part of $A$. Then, we have
\begin{equation*}
	\tau(\HSpin(2n))=[\hat{R}: R]_{0},\,\,\,\,	\tau(\Spin(2n))=[\hat{R}: \tilde{R}]_{0},\,\,\,\, \tau(\SO(2n))=[\hat{R}: R']_{0},
\end{equation*}
and $\tau(\PGO^{+}(2n))=[\hat{R}: \bar{R}]_{0}$.
Note that by \cite{Totaro_E8} and \cite{Totaro_Spin}, we have
\begin{equation*}
	\tau(\SO(2n))=2^{n-1},\,\,\,\,	\frac{\tau(\PGO^{+}(2n))}{\tau(\SO(2n))}=\begin{cases} 1  & \text{if $n$ is not a power of $2$},\\
		2 & \text{otherwise,}
	\end{cases}
\end{equation*}
and 
\begin{equation*}
	\frac{\tau(\SO(2n))}{\tau(\Spin(2n))}=2^{\big\lfloor \log_2\big({n\choose 2}+1 \big)\big\rfloor} \text{ or } 2^{\big\lfloor\log_2\big({n\choose 2}+1 \big)\big\rfloor-1}.
\end{equation*}
From this point onward, we will no longer use the notation $y$ as introduced earlier. In the following, we may reuse $y$ for different purposes.

We present an elementary upper bound for the torsion index of $\HSpin(2n)$:
\begin{lemma}\label{lem:s22mchoose2}
	For any even integer $n\geq 4$, we have
	\begin{equation*}
	\tau_{2}(\HSpin(2n))\leq \begin{cases} n-S_{2}\big({n\choose 2}\big) & \text{ if } v_{2}(n)=1,\\
	n-S_{2}\big({n\choose 2}\big)+1 & \text{ if } v_{2}(n)\geq 2,\\
	\end{cases}
	\end{equation*}
	where $S_{2}({n\choose 2})$ denotes the sum of the base-$2$ digits of ${n\choose 2}$.
\end{lemma}
\begin{proof}
	Consider the following element of $R$:
	\begin{equation*}
	x=\begin{cases}
		e_{1}^{\frac{n(n-1)}{2}}\cdot d_{n-1}& 	\text{ if } v_{2}(n)=1,\\
		e_{1}^{\frac{n(n-1)}{2}-1}\cdot d_{n-1}\cdot 2t & 	\text{ if } v_{2}(n)\geq 2.
		\end{cases}
	\end{equation*}
	By \cite[Lemma 3.3]{Totaro_Spin}, we have 
	\begin{equation*}
	e_{1}^{\frac{n(n-1)}{2}}=2^{n-1-S_{2}({n\choose 2})}\cdot a\cdot \prod_{i=1}^{n-1}e_{i}=:y
	\end{equation*}
	for some odd integer $a$. If $v_{2}(n)=1$, then  by (\ref{chernclassd}),
	\begin{equation*}
x=y\cdot \big({n\choose n-1}t^{n-1}-\cdots -2e_{n-1} \big)=b\cdot 2y\cdot  t^{n-1}=2^{n-S_{2}({n\choose 2})}\cdot ab\cdot x_{0}
	\end{equation*}
for some odd integer $b$ (here we use the relation that monomials in $e_i$s of degree higher than $n(n-1)/2$ are zero), and thus the inequality follows.

	 Now assume $v_{2}(n)\geq 2$.  By (\ref{relationtandc}) and (\ref{chernclassd}),
	\begin{equation*}
		\begin{split}
			d_{n-1}\cdot 2t 
			& = 2{n\choose n-1}t^{n}+2 \big({n-1\choose n-2}t^{n-1}c_{1}+\cdots+ c_{n-1}t \big) \\
			& = 2{n\choose n-1}(-c_{1}t^{n-1}-\cdots -c_{n-1}t)+2 \big({n-1\choose n-2}t^{n-1}c_{1}+\cdots+ c_{n-1}t \big)\\
			& = -2c_{1}t^{n-1}+2\sum_{i=2}^{n-1}a_{i}c_{i}t^{n-i}= 4 \big( e_1 t^{n-1} + \sum_{i=2}^{n-1} (-1)^i a_i e_i t^{n-i})
		\end{split}
	\end{equation*}
	for some $a_i \in \Z$. Hence, we obtain
	\begin{equation*}
			x=4 e_1^{\frac{n(n-1)}{2}-1} \big( e_1 t^{n-1} + \sum_{i=2}^{n-1} (-1)^i a_i e_i t^{n-i})=4 y\cdot t^{n-1}=2^{n-S_{2}({n\choose 2})+1}\cdot a\cdot x_{0},
	\end{equation*}
	and thus the inequality follows.
\end{proof}

\subsection{Some relations in the subrings of $\CH(X)$}\label{subsec4:two-step}
In this section we discuss five relations: (\ref{eq:t2nminusone}), (\ref{eq:tnconeprime}), Lemma \ref{lem:expandtlarge2}, (\ref{eq:twoniti}), and (\ref{eq:disquared}).

Let $E$ be the tautological $n$-dimensional vector bundle on $Y$, and let $L=\mathcal{O}_{\mathbb{P}(E)}(1)$ be the canonical line bundle on $X=\mathbb{P}(E)$ as defined in Section \ref{subsec:two-step}. Consider the projection $\pi: X=\mathbb{P}(E)\to Z$, where $Z$ is the projective quadric of $q$. Then, 
\begin{equation*}
	\dim Z=2n-2 \text{ and } L^{*}=\mathcal{O}_{\mathbb{P}(E)}(-1)=\pi^{*}(\mathcal{O}_{Z}(-1)).
\end{equation*}
As $c_{1}(\mathcal{O}_{Z}(-1))^{2n-1}=0$ and $t=c_{1}(L)$, we have the following relation
\begin{equation}\label{eq:t2nminusone}
	t^{2n-1}=0 \text{ in } R'\subset \tilde{R}\subset \hat{R}.
\end{equation}

Recall the following two formulas:
\begin{equation*}
c_{i}(\mathcal{E}\tens \mathcal{L})=\sum_{j=0}^{i}{n-j\choose i-j}c_{1}(\mathcal{L})^{i-j}c_{j}(\mathcal{E})
\end{equation*}
for an $n$-dimensional vector bundle $\mathcal{E}$ and a line bundle $\mathcal{L}$ on a variety $V$, and
\begin{equation*}
c_{i}(\mathcal{E}\oplus \mathcal{F})=\sum_{j=0}^{i}c_{j}(\mathcal{E})c_{i-j}(\mathcal{F})
\end{equation*}
for a vector bundle $\mathcal{F}$ on $V$.

It follows from the relation $c_{n}(E)=0$ that
\begin{equation}\label{eq:tncone}
	\begin{split}
0=c_{n}(E)&=c_{n}(E\tens L\tens L^{*})\\
&=\sum_{j=0}^{n}c_{1}(L^{*})^{n-j}c_{j}(E\tens L)\\
&=(-t)^{n}+(-t)^{n-1}d_{1}+\cdots + (-t)d_{n-1},
	\end{split}
\end{equation}
where $d_{i}=c_{i}(E\tens L)$ and $c_{1}(L)=t$. In particular, since $n$ is even, the following relation holds in $R'$:
\begin{equation}\label{eq:tnconeprime}
	t^{n}-d_{1}t^{n-1}+d_{2}t^{n-2}+\cdots -d_{n-1}t=0.
\end{equation}

\begin{lemma}\label{lem:expandtlarge2}
	For each integer $1 \le k \le n-1$, 
	$$
	t^{n-1+k} = t^{n-1} c_k + \sum_{i=1}^{n-2} t^i a_i
	$$
	in $R'/2R'$, where each $a_i$ is a homogeneous polynomial in $c_{1},\ldots, c_{n-1}$ of total degree $n-1+k-i$.
\end{lemma}
\begin{proof}
	Denote $c(i_1, \ldots, i_s) = c_{i_1} \ldots c_{i_s}$. Let $S_{k,j}$ be defined as
	\[
	S_{k,j} = \sum_{\substack{i_1 \ge j \\ i_1 + \cdots + i_s = k}} c(i_1,\ldots, {i_s}),
	\]
	where the summation is taken over all tuples $(i_1, \ldots, i_s)$ of natural numbers. Then, for any $1\leq j\leq n-2$, we have
	\begin{equation}\label{eq:Skjdifference}
		S_{k+j,j}-S_{k+j,j+1}=c_{j}S_{k,1}
	\end{equation}
	as both expressions represent the sum of $c(i_1,\ldots, {i_s})$ such that $i_{1}=j$ and $i_1 + \cdots + i_s = k+j$. Similarly, for $j=n-1$ we have
	\begin{equation}\label{eq:Skjdifferenceprime}
		S_{k+n-1,n-1}=c_{n-1}S_{k,1},
	\end{equation}
	since $c_{j}=0$ for any $j\geq n$.
	
	We prove by induction on $1\leq k\leq n-1$ that the following equation
	\begin{equation}\label{eq:formofpoweroft}
		t^{n-1+k} = t^{n-1} S_{k,1} + t^{n-2} S_{k+1, 2} + \cdots + t S_{k+n-2,n-1} 
	\end{equation}
	holds modulo $2R'$. For $k=1$, since $S_{j, j}=c_{j}$, the equation follows immediately from the relation in (\ref{relationtandc}). In general, let us multiply (\ref{eq:formofpoweroft}) by $t$ and expand it using the relation in (\ref{relationtandc}). Then, by (\ref{eq:Skjdifference}) and (\ref{eq:Skjdifferenceprime}), we get (still modulo $2R'$)
	\begin{align*}
		t^{n-1+k+1}&= t^n S_{k,1} + t^{n-1} S_{k+1, 2} + \cdots + t^2 S_{k+n-2,n-1}\\
		&= (t^{n-1}c_{1}+\cdots +tc_{n-1})S_{k,1}+t^{n-1} S_{k+1, 2} + \cdots + t^2 S_{k+n-2,n-1}\\
		&= t^{n-1}(c_{1}S_{k,1}+S_{k+1,2})+\cdots +t^{2}(c_{n-2}S_{k,1}+S_{k+n-2,n-1}) +tc_{n-1}S_{k,1}\\
		&= t^{n-1}S_{k+1,1}+\cdots + t^{2}S_{k+n-2,n-2}+tS_{k+n-1,n-1},
	\end{align*} 
	which completes the induction.
	
	Finally, since $c_{j}^{2}\equiv 0 \mod 2R'$ for any $j$, we have
	\begin{equation}
		\sum_{\substack{s\geq 2 \\ i_1 + \cdots + i_s = k}} c(i_1, \ldots, i_s)\equiv \sum_{\substack{s\geq 2,\,\, i_{1}<\cdots<i_{s} \\ i_1 + \cdots + i_s = k} } s!\cdot c(i_1, \ldots, i_s)\equiv 0 \mod 2R',
	\end{equation} 
	thus $S_{k,1}\equiv c_{k} \mod 2R'$. Hence, the statement follows from (\ref{eq:formofpoweroft}).\end{proof}

Since $c_{i}(E)=(-1)^{i}2e_{i}$ and $c_{i}(E^{*})=(-1)^{i}c_{i}(E)$, it follows from the relation (\ref{chowringY}) that
\begin{equation*}
	c_{2i}(E\oplus E^{*})=\sum_{j=0}^{2i}c_{j}(E)c_{2i-j}(E^{*})=(-1)^i \cdot 4\big(	e_{i}^{2}-2e_{i-1}e_{i+1}+\cdots +(-1)^{i}e_{2i}\big)=0
\end{equation*}
for any $i>0$. Also, we have
\begin{align*}
	c_{2i+1}(E\oplus E^{*})&=
	\sum_{j=0}^{2i+1}c_{j}(E)c_{2i+1-j}(E^{*})
	\\
	&=\sum_{j=0}^{i+1}(-1)^{2i+1-j}c_{j}(E)c_{2i+1-j}(E)+\sum_{k=0}^{i+1}(-1)^k c_{2i+1-k}(E)c_{k}(E)=0
\end{align*}
for any $i \ge 0$. Therefore, we get
\begin{equation}\label{eq:cieestar}
c_{i}\big((E\oplus E^{*})\tens L \big)={2n\choose i}c_{1}(L)^{i}={2n\choose i}t^{i}.
\end{equation}

On the other hand, as $\bar{R}=\langle 2t, d_{i}\rangle$ and $c_{1}(L\tens L)=2t$, 
it follows from
\begin{equation*}
E^{*}\tens L=(E\tens L)^{*}\tens (L\tens L)
\end{equation*}
that
\begin{equation}\label{eq:cieplusestar}
c_{i}\big((E\oplus E^{*})\tens L \big)=c_{i}\big((E\tens L)\oplus ((E\tens L)^{*}\tens (L\tens L)) \big)\in \bar{R}.
\end{equation}
Hence, by (\ref{eq:cieestar}) we have
\begin{equation}\label{eq:twoniti}
	{2n\choose i}t^{i}\in \bar{R}
\end{equation}
for $i\geq 1$.

Finally, since we have
\begin{equation*}
	c_{i}\big((E\tens L)^{*}\tens (L\tens L) \big)=c_{i}\big( (E\tens L)^{*}   \big)=c_{i}( E\tens L)=d_{i} \text{ in } R'/2R',
\end{equation*}
it follows from (\ref{eq:cieplusestar}) that
\begin{equation*}
c_{2i}\big((E\oplus E^{*})\tens L \big)=d_{i}^{2} \text{ in } R'/2R'.
\end{equation*}
Therefore, by (\ref{eq:cieestar}) and by Kummer's theorem we get
\begin{equation}\label{eq:disquared}
	d_{i}^{2}={2n\choose 2i}t^{2i}={n\choose i}t^{2i} \text{ in } R'/2R'
\end{equation}
for all $1\leq i\leq n-1$.

\section{Main results}
In this section, we provide the proof of Theorem \ref{thm:main}, which is established in Corollaries \ref{cor:intervalupperbd}, \ref{cor:2powerthreeupp}, and \ref{cor:3powerthreeupp}. 
The key aspect of the proof (addressed in Proposition \ref{prop:totaro2divisible}) is to detect an element of $R$ of the top degree that has the same $2$-divisibility in $\hat{R}$ as an element of $\tilde{R}$ of top degree. Here, by the \emph{2-divisibility} of an element $r \in \hat{R}$ in $\hat{R}$ we mean the biggest integer (denoted $\hat{v}_2(r)$) such that $r$ is divisible by $2^{\hat{v}_2(r)}$ in $\hat{R}$.

Throughout this section, we fix an even integer $n$. For any subset $I$ of $[1,n-1]$, we define
\begin{equation*}
	c(I)=\prod_{i\in I}c_{i},\quad e(I)=\prod_{i\in I}e_{i},\quad d(I)=\prod_{i\in I}d_{i}.
\end{equation*}

\begin{dfn}\label{def:totarodecomp}
	Let $J \subset [1,n-1]$ be a subset. The sum of all elements in $J$ is called the \emph{degree} of $J$ and is denoted by $\operatorname{deg}J$. We say that $J$ is \emph{Totaro-decomposable} if $\operatorname{deg}J=2^a-1$ for some integer $a\geq 1$ and 
	$J$
	can be decomposed into a
	disjoint union of powers of $2$ and pairs with the sum of elements equal to a power of $2$. A subset $J$ is called \emph{strongly Totaro-decomposable} if it is Totaro-decomposable and $$J\cap [1,2^{v_{2}(n)}]=\{2^{i}\,|\, i\in [0,v_2(n)]\}\cap [1,n-1  ].$$
\end{dfn}

We will use the following result, which is essentially the same as \cite[Lemma 5.4]{Totaro_Spin}.

\begin{lemma}\cite[\S 4, 5]{Totaro_Spin}\label{lem:totarodecomp}
	Let $J$ be a subset of $[1,n-1]$ such that $\deg J= 2^a-1$ for some integer $a\geq 1$ and let $I=[1,n-1]\setminus J$. Then, 
	\begin{equation}\label{eq:totarodivisibilitypow2_prime}
			\hat{v}_{2}\big(e_{1}^{\deg J}c(I)t^{n-1}\big)=n-a-1
	\end{equation}
	if $J$ is Totaro-decomposable, and
	\begin{equation}\label{eq:totarodivisibilitypow2_prime2}
		\hat{v}_{2}\big(e_{1}^{\deg J}c(I)t^{n-1}\big)\geq n-a
	\end{equation}
	otherwise.
\end{lemma}
\begin{proof}
By using (\ref{eq:relationcitwo}), write
\begin{equation*}
c_{1}^{\deg J}=\sum b(J')c(J')	
\end{equation*}
for some integers $b(J')$, where the sum ranges over all subsets $J'\subset [1,n-1]$ with $\deg J'=\deg J$. (For more details on how to get square-free monomials using (\ref{eq:relationcitwo}), see \cite[Proof of Proposition 86.16]{EKM}.) Then, by \cite[Lemma 5.4]{Totaro_Spin}
\begin{equation}\label{ordercheckv22}
	v_{2}(b(J'))=\deg J -a 
\end{equation}
if $J'$ is Totaro-decomposable, and
\begin{equation}\label{ordercheckv33}
	v_{2}(b(J'))\geq \deg J -a +1
\end{equation}
otherwise.

Since $I\cap J'\neq \emptyset$ for any $J'\neq J$, by (\ref{eq:relationcitwo}) we have
\begin{equation*}
c_{1}^{\deg J}c(I)=b(J) c([1,n-1])+2\sum_{J'\neq J}b(J')k(J')c([1,n-1])
\end{equation*}
for some integers $k(J')$. As $c_{1}=-2e_{1} $ and  
$c([1,n-1])=\pm 2^{n-1}e([1,n-1])$,
the equations (\ref{eq:totarodivisibilitypow2_prime}) and (\ref{eq:totarodivisibilitypow2_prime2}) follow from (\ref{ordercheckv22}) and (\ref{ordercheckv33}). Note that the multiplication by the term $t^{n-1}$ does not affect on the $2$-divisibility of the element $e_{1}^{\deg J}c(I)$ in $\hat{R}$. \end{proof}

\begin{remark}\label{rmk:totarodecomparbitraryY}
In this lemma, instead of $c(I)t^{n-1}$, we can take an a priori arbitrary $y \in R'$ such that $e_1^{\deg J} y \in \tilde{R}$ 
is homogeneous of top degree and still get $\hat{v}_2 (e_1^{\deg J} y) \ge n-a-1$. 
Indeed, starting with such an arbitrary $y$, first, we can assume without loss of generality, using (\ref{relationtandc}) and (\ref{eq:relationcitwo}), 
that $y = c(I') t^j$, where $I'$ is an arbitrary subset of $[1,n-1]$ and $0 \le j \le n-1$. 
Second, if $j < n-1$, then $e_1^{\deg J} c(I')$ is a monomial in $\CH(Y)$ of degree bigger than $\dim Y$, i.e. $e_1^{\deg J} c(I')=0$.
And if $j=n-1$, then $\deg I'+\deg J = \frac{n(n-1)}2$, and we get back exactly the situation of Lemma \ref{lem:totarodecomp} itself.
\end{remark}

\begin{remark}\label{rmk:generalJ}
	Let $v'_{2}(r)$ denote the $2$-divisibility of an element $r\in R'$ in $R'$. As shown in \cite[\S4]{Totaro_Spin}, 
	\[
	v'_{2}(c_{1}^{\deg J}) \geq \deg J - S_{2}(\deg J),
	\]  
	for any subset $J$ of $[1,n-1]$, where $S_{2}(\deg J)$ denotes the sum of the base-$2$ digits of $\deg J$ (Note that in \cite[\S4]{Totaro_Spin}, the $2$-divisibility is computed not in $R'$ but in the ring generated by the Chern classes $c_1, \ldots, c_{n-1}$. However, for monomials involving only the $c_i$, such as $c_1^{\deg J}$, the $2$-divisibility remains the same.) Thus, for $I = [1, n-1] \setminus J$, we obtain  
	\[
	v'(c_{1}^{\deg J} c(I) t^{n-1}) \geq \deg J - S_{2}(\deg J).
	\]  
	Then, as in the arguments of \cite[\S4]{Totaro_Spin}, the top-degree element $c_1^{\deg J} c(I) t^{n-1}$ is a multiple of $2^{\deg J - S_2(\deg J)} c([1,n-1]) t^{n-1}$. In $\hat{R}$, since $c_1 = -2e_1$, it follows that $2^{\deg J} e_1^{\deg J} c(I) t^{n-1}$ is a multiple of  
	\[
	2^{n-1+\deg J - S_2(\deg J)} e([1,n-1]) t^{n-1}.
	\]  
	Thus, we obtain  
	\[
	\hat{v}(e_1^{\deg J} c(I) t^{n-1}) \geq n-1 - S_2(\deg J).
	\]  
\end{remark}

\smallskip

The following result can be viewed as an application of Lemma \ref{lem:totarodecomp} under the assumption of a strongly Totaro-decomposable subset. In the next corollary, we apply this result to replace an element $d(I')$ in $R$ with $c(I')$ in the larger ring $\tilde{R}$ for a specific subset $I'$.

\begin{lemma}\label{lem:totarodecomposableab}
	Let $q\leq v_{2}(n)$ be a nonnegative integer. Let $J$ be a strongly Totaro-decomposable subset of $[1,n-1]$ with $I=[1,n-1]\backslash J$. Assume that $I'$ is a subset of $[1,n-1]$ containing $[1,2^{q}]\setminus \{2^{i}\,\,|\,\, i\in [0,q-1]\}$ such that $e_{1}^{\deg J}c(I')y$ is an element of top degree in $\tilde{R}$ for some homogeneous element $y\in R'$. Then, we have
	$$
e_{1}^{\deg J}\cdot c(I')\cdot y \equiv 0 \mod 2^{p+1}\hat{R},
	$$
		where $p=\hat{v}_2(e_1^{\deg J} c(I) t^{n-1})$.
\end{lemma}
\begin{proof}
Let $x=e_{1}^{\deg J}c(I')y$. It suffices to consider the case where $y$ is a monomial in $c_{1},\ldots, c_{n-1}$, and $t$. Furthermore, since the element $x$ is of top degree in $\tilde{R}$, we may assume that 
\begin{equation*}
x=e_{1}^{\deg J}c(I')t^{n-1}
\end{equation*}
for some subset $I'$ of $[1,n-1]$ containing $[1,2^{q}]\setminus \{2^{i}\,|\, i\in [0,q-1]\}$, where $\deg I'=\deg I$.

	Let $J'=[1,n-1]\setminus I'$. Then, $\deg J=\deg J'$ and
	\begin{equation}\label{eq:Jprimesubset}
		J'\cap [1,2^{q}]\subseteq\{2^{i}\,|\, i\in [0,q-1]\}, 
	\end{equation}
	thus $\deg(J'\cap [1,2^{q}])\leq 2^{q}-1$.  Assume that $J'$ is Totaro-decomposable. If $J'$ contains a pair whose sum is a power of $2$, then by (\ref{eq:Jprimesubset}), this sum is divisible by $2^{q+1}$. Hence, $\deg J'$ is congruent to the degree of a subset of $J'\cap [1,2^{q}]$ modulo $2^{q+1}$, that is, 
	\begin{equation}\label{eq:degreeJprimem}
	\deg J'\equiv m \mod 2^{q+1}
	\end{equation}
	 for some integer $0\leq m \leq 2^{q}-1$.
	
	On the other hand, since $J$ is strongly Totaro-decomposable, a similar argument shows that $\deg J\equiv -1 \mod 2^{v_{2}(n)+1}$, thus  $\deg J'\equiv -1 \mod 2^{q+1}$, which contradicts (\ref{eq:degreeJprimem}). Hence,  $J'$ is not Totaro-decomposable. Therefore, the statement follows from Lemma \ref{lem:totarodecomp}.\end{proof}

\begin{remark}\label{rem:almostoddnstronglytotarodec}
Let $J \subseteq [1,n-1]$ be a Totaro-decomposable subset. 
A similar but simpler argument using residues modulo $4$ of $\deg J$ and of sums of pairs of elements of $J$ shows that 
if $\deg J > 1$, then $1, 2 \in J$. 
Therefore, if $\deg J > 1$ and $v_2(n)=1$, then $J$ also satisfies the definition of a strongly Totaro-decomposable subset.
\end{remark}

\begin{corollary}\label{cor:totarodecomposableab}
Let $q$, $p$, $J$ and $I$ be as in Lemma \ref{lem:totarodecomposableab}. 
	Let $I'=[1,2^{q}]\setminus \{2^{i}\,\,|\,\, i\in [0,q]\}$. Assume that $e_{1}^{\deg J}c(I')y$ is an element of top degree in $\tilde{R}$ for some homogeneous element $y\in R'$. 
	Then, in $\hat{R}/2^{p+1}\hat{R}$, we have 
	$$
	e_{1}^{\deg J}\cdot d(I')\cdot y=e_{1}^{\deg J}\cdot c(I')\cdot y.
	$$
\end{corollary}
\begin{proof}
We may assume $q \ge 2$, otherwise $I'$ is empty. For each $k\in I'$, we replace a factor $d_{k}$ of $d(I')$ by $c_{k}$ using the relation (\ref{chernclassd}) as follows. As a first step, we expand the following:
\begin{equation}\label{eq:dionenchoose3}
	d(I')=\big({n\choose 3}t^{3}+{n-1\choose 2}t^{2}c_{1}+{n-2\choose 1}t^{1}c_{2}+c_{3}\big)\cdot d(I'\setminus \{3\}).
\end{equation}
Since, according to Kummer's theorem, the coefficient ${n\choose 3}$ is divisible by $2$, by Lemma \ref{lem:totarodecomp} (with Remark \ref{rmk:totarodecomparbitraryY}), the first term of (\ref{eq:dionenchoose3}) multiplied by $e_{1}^{\deg J}y$ vanishes in $\hat{R}/2^{p+1}\hat{R}$. Since the second and third terms in (\ref{eq:dionenchoose3}) have $c_{1}$ and $c_{2}$ as factors, respectively, by Lemma \ref{lem:totarodecomposableab}, the following relation $$
e_{1}^{\deg J}\cdot d(I')\cdot y=e_{1}^{\deg J}\cdot c_{3}\cdot d(I'\setminus \{3\})\cdot y.
$$
holds in $\hat{R}/2^{p+1}\hat{R}$.

Assume that for some $k \in I'$, $k > 3$, the relation $e_{1}^{\deg J}\cdot d(I')\cdot y=e_{1}^{\deg J}\cdot c(I'')\cdot d(I'\setminus I'')\cdot y$ holds in $\hat{R}/2^{p+1}\hat{R}$, where $I''$ denotes the subset of $I'$ consisting of all consecutive elements less than $k$, starting from $3$. Let $z=d(I'\backslash ( I''\cup \{k\} ) )\cdot y$. Similarly, using the relation (\ref{chernclassd}), we expand:
\begin{equation}\label{expandcIprime}
	c(I'')\cdot d_{k}=c(I'')\big({n\choose k}t^{k}+{n-1\choose k-1}t^{k-1}c_{1}+\cdots+{n-k+1\choose 1}t^{1}c_{k-1}+c_{k}\big).
\end{equation}
Since the coefficient ${n\choose k}$ is divisible by $2$ due to Kummer's theorem, it follows from Lemma \ref{lem:totarodecomp} and Remark \ref{rmk:totarodecomparbitraryY} that the first term of $(\ref{expandcIprime})$, when multiplied by $e_{1}^{\deg J}z$, vanishes in 
$\hat{R}/2^{p+1}\hat{R}$. 
Similarly, since $I''=[1,k-1]\setminus \{2^{i}\mid i\in \N\cup \{0\} \}$, every term of (\ref{expandcIprime}) containing $c(I'')\cdot c_{j}$, where $j < k$ and $j$ is not a power of $2$, has a factor of $c_{j}^{2}$. Thus, by (\ref{eq:relationcitwo}), each of these terms is divisible by $2$ in $R'$. Consequently, when multiplied by $e_{1}^{\deg J}z$, each such term vanishes in 
$\hat{R}/2^{p+1}\hat{R}$. 
Finally, by Lemma \ref{lem:totarodecomposableab}, every term of (\ref{expandcIprime}) containing $c(I'')\cdot c_{i}$, where $i$ is a power of $2$, also vanishes in 
$\hat{R}/2^{p+1}\hat{R}$ 
when multiplied by $e_{1}^{\deg J}z$, thereby completing the proof by induction.
\end{proof}

We now provide a key ingredient in the proof of Theorem \ref{thm:main}, where we establish a top-degree element in $R$ whose $2$-divisibility matches exactly with that of a top-degree element in $\tilde{R}$, which is computable.

\begin{proposition}\label{prop:totaro2divisible}
Let $m$ be a divisor of $2^{v_{2}(n)}$ such that $2\leq m \leq n/2$. Let $J$ be a strongly Totaro-decomposable subset of $[1,n-1]$ with $I=[1,n-1]\backslash J$. Assume that $I$ contains $[n-m + 1, n - 1]$. Then,
\begin{equation*}
	\hat{v}_{2}(e_{1}^{\deg J}\cdot d(I \setminus \{n - m + 1\}) \cdot t^{2n-m} ) = p,
\end{equation*}
where $p=\hat{v}_2(e_1^{\deg J} c(I) t^{n-1})$. 
Moreover, if $m=2^{v_{2}(n)}$, then
\begin{equation*}
	\hat{v}_{2}(e_{1}^{\deg J}\cdot d(I \setminus \{n - m + 1\}) \cdot d_{m}^{\,\,\frac{2n - m}{m}} )=p.
\end{equation*}
\end{proposition}
\begin{proof}
Let $x=e_{1}^{\deg J} \cdot d(I \setminus \{n - m + 1\}) \cdot t^{2n-m}$. Write
\begin{equation*}
	I=([1,m]\setminus \{2^{i}\,\,|\,\, i\in [0,v_2(m)]\,\})\cup (I\cap [m+1,n-1])=:I_{1}\cup I_{2}.
\end{equation*}
Then, by Corollary \ref{cor:totarodecomposableab}, we have
\begin{equation}\label{eq:dcreplacement}
	x=e_{1}^{\deg J}\cdot c(I_{1})\cdot d(I_{2}\setminus \{n-m+1\})\cdot t^{2n-m}
\end{equation}
in $\hat{R}/2^{p+1}\hat{R}$.

By (\ref{eq:t2nminusone}) and (\ref{eq:tnconeprime}), for any $m \ge 4$, the product of last two terms in (\ref{eq:dcreplacement}) becomes
\begin{equation}\label{eq:dicaptwop}
	d(I_{2}\setminus \{n-m+1, n-1\})\cdot (d_{n-m+1}t^{2n-2}+\cdots +d_{n-2}t^{2n-m+1})
\end{equation}
in $R'/2R'$. Since, by (\ref{eq:t2nminusone}) and (\ref{eq:disquared}),  
\begin{equation*}
	d_{n-m+j}^{2}t^{2n -1 -j}={n \choose n-m+j}t^{4n-2m+j -1}=0
\end{equation*}
in $R'/2R'$ for all $2\leq j \leq m-2$, the expression
(\ref{eq:dicaptwop}) simplifies to
\begin{equation}\label{eq:dicaptwopower}
	d(I_{2}\setminus \{n-1\})\cdot t^{2n-2}=c(I_{2}\setminus \{n-1\})\cdot t^{2n-2}
\end{equation}
in $R'/2R'$, where the equality in (\ref{eq:dicaptwopower}) follows from (\ref{chernclassd}) and (\ref{eq:t2nminusone}). Therefore, by Lemma \ref{lem:totarodecomp} and Remark \ref{rmk:totarodecomparbitraryY},
\begin{equation}\label{eq:dicaptwopowerprime}
	x=e_{1}^{\deg J}\cdot c(I_{1})\cdot c(I_{2}\setminus \{n-1\})\cdot t^{2n-2}
\end{equation}
in $\hat{R}/2^{p+1}\hat{R}$. Note that for $m=2$, the equality in (\ref{eq:dicaptwopowerprime}) follows directly from (\ref{chernclassd}) and (\ref{eq:t2nminusone}).

By Lemma \ref{lem:expandtlarge2}, we have
\begin{equation*}
	t^{2n-2}=t^{n-1}c_{n-1}+\sum_{i=1}^{n-2}t^{i}a_{i}	
\end{equation*}
in $R'/2R'$, where each $a_{i}$ is a homogeneous polynomial in $c_{1},\ldots, c_{n-1}$ of total degree $2n-2-i$, which is greater  than $n-1$ for any $1\leq i\leq n-2$. As 
\begin{equation*}
	\tfrac{n(n-1)}{2}=\deg(e_{1}^{\deg J}c(I))<\deg(e_{1}^{\deg J}c(I\setminus \{n-1\})\cdot a_{i})
\end{equation*}
in $\CH(Y)$, we have $e_{1}^{\deg J}c(I\setminus \{n-1\})\cdot a_{i}=0$ for each $1\leq i\leq n-2$. Hence, 
\begin{equation}\label{eq:edegJprime}
	e_{1}^{\deg J}\cdot c(I_{1})\cdot c(I_{2}\setminus \{n-1\}\})\cdot (t^{n-1}c_{n-1}+\sum_{i}t^{i}a_{i})=e_{1}^{\deg J}\cdot c(I)\cdot t^{n-1}
\end{equation}
in $\tilde{R}$. Therefore, it immediately follows from (\ref{eq:dicaptwopowerprime}) that 
$\hat{v}_{2}(x)=p$.

Now, assume that $m=2^{v_{2}(n)}$. Let $x'=e_{1}^{\deg J}\cdot d(I \setminus \{n - m + 1\}) \cdot d_{m}^{\,\,\frac{2n - m}{m}}$. Then, by Corollary \ref{cor:totarodecomposableab},
\begin{equation}\label{eq:xprimeeoneJ}
	x'=e_{1}^{\deg J}\cdot c(I_{1})\cdot d(I_{2}\setminus \{n-m+1\})\cdot d_{m}^{\,\,\frac{2n - m}{m}}
\end{equation}
in $\hat{R}/2^{p+1}\hat{R}$. Using the relation (\ref{chernclassd}), we expand as follows:
\begin{equation*}
	c(I_{1})\cdot d_{m}=c(I_{1})\big({n\choose m}t^{m}+{n-1\choose m-1}t^{m-1}c_{1}+\cdots+{n-m+1\choose 1}t^{1}c_{m-1}+c_{m}\big).
\end{equation*}
By Kummer's theorem, the coefficient ${n\choose m}$ of the term containing $t^{m}$ is odd. Thus, by the same reasoning as in the proof of Corollary \ref{cor:totarodecomposableab},  along with Lemmas \ref{lem:totarodecomp} (with Remark \ref{rmk:totarodecomparbitraryY}) and \ref{lem:totarodecomposableab}, we obtain the following relation from (\ref{eq:xprimeeoneJ}):
\begin{equation}\label{eq:edegJdtwop}
	x'=e_{1}^{\deg J}\cdot c(I_{1})\cdot d(I_{2}\setminus \{n-m+1\})\cdot t^{2n-m}
\end{equation}
in 
$\hat{R}/2^{p+1}\hat{R}$. 
Hence, from (\ref{eq:dcreplacement}), it follows that 
$\hat{v}_{2}(x')=\hat{v}_{2}(x)=p$.
\end{proof}

In the following three corollaries, we present the proof of Theorem \ref{thm:main}.

\begin{corollary}\label{cor:intervalupperbd}
Let $n$ be an even integer in $(2^{s}, 2^{s}+2^{s-1})\cup (2^{s}+2^{s-1}, 2^{s+1})$ for some integer $s\geq 3$. Then, we have
\begin{equation}\label{eq:inequalityHspin}
	\tau(\HSpin (2n))\leq 2\cdot\tau(\Spin (2n)).
\end{equation}
Moreover, we have
\begin{equation}\label{eq:equalityHspin}
\tau(\HSpin (2n))=\tau(\Spin (2n))
\end{equation}
for any $n\in (2^{s}, 2^{s}+m_{0}]\cup (2^s + 2^{s-3}, 2^s + 2^{s-2}) \cup (2^{s}+2^{s-2}, n_{0})\cup  (2^{s}+2^{s-1}, 2^{s+1})$, where $n_{0}$ and $m_{0}$ denote the integers in $(\ref{eq:defnzero})$ and $(\ref{eq:defmzero})$.
\end{corollary}	
\begin{proof}
Let $n\in (2^{s}, 2^{s+1})\backslash \{3\cdot 2^{s-1}\}$ 
be an even integer. The proof is divided into three cases. In each case, we choose a strongly Totaro-decomposable subset $J$ with $I:=[1,n-1]\setminus J$. Let
\begin{equation*}
	J_{s}:=\{2^{s-1}+2^{s-2}-1, 2^{s-1}-(2^{s-2}-1)\}\cup \cdots \cup \{2^{s-1}+1, 2^{s-1}-1\}\cup \{2^{i}\,|\, i\in [0,s-1]\}
\end{equation*}
denote a  strongly Totaro-decomposable subset with $\deg(J_{s})=2^{s}(2^{s-2}-1)+2^{s}-1=2^{2s-2}-1$. Note that this set comes from \cite[Lemma 6.2]{Totaro_Spin}.

\bigskip

\noindent{\it Case 1.} Let $n\in (2^{s}, 2^{s}+2^{s-3}]\cup \{2^{s}+2^{s-2}\}$. Set $J=J_{s}$. As $v_{2}(n)\leq s-2$, we have $n-2^{v_{2}(n)}+1>2^{s}-2^{s-2}=2^{s-1}+2^{s-2}$, which is bigger than any element of $J_{s}$. Hence, by Proposition \ref{prop:totaro2divisible} and Lemma \ref{lem:totarodecomp},
\begin{equation*}
\tau_{2}(\HSpin (2n))\leq n-2s+1.
\end{equation*}
Therefore, by Theorem \ref{thm:totaromain} the equality (\ref{eq:equalityHspin}) holds for $n\in (2^{s}, 2^{s}+m_{0}]$ and the inequality in (\ref{eq:inequalityHspin}) holds for $n\in (2^{s}+m_{0}, 2^{s}+2^{s-3}]\cup \{2^{s}+2^{s-2}\}$.

\bigskip

\noindent{\it Case 2.} Let $n\in (2^{s}+2^{s-3},\, 2^{s}+2^{s-1})\setminus \{2^{s}+2^{s-2}\}$. Set
\begin{equation*}
	J=J_{s}\cup \{2^{s}+2^{s-3},2^{s}-2^{s-3}\}\cup \cdots \cup \{2^{s}+1,2^{s}-1\},
\end{equation*}
where this set is derived from \cite[Lemma 6.1]{Totaro_Spin}. Then, $\deg(J)=\deg(J_{s})+2^{s+1}\cdot 2^{s-3}=2^{2s-1}-1$. If $n\in (2^{s}+2^{s-3}, 2^{s}+2^{s-2})$, then $v_{2}(n)<s-3$, thus we have a binary expression $n=2^{s}+2^{s-3}+\cdots +2^{v_{2}(n)}$, i.e., $n-2^{v_{2}(n)}+1> 2^{s}+2^{s-3}$, which is the biggest element of $J$. Similarly, if $n\in (2^{s}+2^{s-2},2^{s}+2^{s-1})$, then $v_{2}(n)< s-2$, thus $n-2^{v_{2}(n)}+1>2^{s}+2^{s-2}$. Therefore, by Proposition \ref{prop:totaro2divisible} and Lemma \ref{lem:totarodecomp},
\begin{equation*}
	\tau_{2}(\HSpin (2n))\leq n-2s,
\end{equation*}
thus by Theorem \ref{thm:totaromain} the equality (\ref{eq:equalityHspin}) holds for $n\in (2^{s}+2^{s-3}, n_{0})\setminus \{2^{s}+2^{s-2}\}$ and the inequality in (\ref{eq:inequalityHspin}) holds for $n\in [n_{0}, 2^{s}+2^{s-1})$. 

\bigskip

\noindent{\it Case 3.} Let $n\in (2^{s}+2^{s-1},\, 2^{s+1})$. Set $J=J_{s+1}$. We have $v_2(n)\leq s - 2$. Again write the binary expansion $n = 2^s + 2^{s-1} + \cdots + 2^{v_2(n)}$. It follows from our earlier argument that $n - 2^{v(n)} + 1$ exceeds $2^s + 2^{s-1} - 1$, the largest element of $J$.
Hence, by Proposition \ref{prop:totaro2divisible} and Lemma \ref{lem:totarodecomp}, we obtain
\begin{equation*}
\tau_{2}(\HSpin (2n))\leq n-2s-1.
\end{equation*}
Therefore, by Theorem \ref{thm:totaromain} the equality (\ref{eq:equalityHspin}) holds.
\end{proof}

\begin{corollary}\label{cor:2powerthreeupp}
	Let $n=3\cdot 2^{s}$ or $n=2^{s}$ for some integer $s\geq 1$. Then, 	\begin{equation}\label{eq:inequalityHspintwossminus}
		\tau(\HSpin (2n))\leq \begin{cases} 2^{s}\cdot \tau(\Spin (2n))\quad \text{ if } s=1, 2,\\  2^{3}\cdot\tau(\Spin (2n))\quad \text{ if } s\geq 3.\end{cases}
	\end{equation}
\end{corollary}
\begin{proof}
Let $n=3\cdot 2^{s}$ for $1\leq s\leq 2$ and set $m=2^{s}$. For each $s$, consider the strongly Totaro-decomposable subset $J=\{2^{i}\,\,|\,\, i\in [0,s+1]\,\}$ and set $I = [1,n-1] \setminus J$. Then, by Proposition \ref{prop:totaro2divisible} and Lemma \ref{lem:totarodecomp}, we obtain 
\begin{equation*}
	\hat{v}_{2}(e_{1}^{\deg J}\cdot d(I \setminus \{n-m+1\}) \cdot d_{m}^{5} ) = \hat{v}_2(e_{1}^{\deg J} c(I) t^{n-1})=\begin{cases} 2 & \text{ if } s=1,\\ 7 & \text{ if } s=2.\end{cases}
\end{equation*}
Since $\tau_{2}(\Spin(2n))=1$ for $s=1$ (i.e., $\Spin(12)$) and $\tau_{2}(\Spin(2n))=5$ for $s=2$ (i.e., $\Spin(24)$) by Theorem \ref{thm:totaromain}, the inequalities in (\ref{eq:inequalityHspintwossminus}) follow.

Let $n=2^{s}$ for $1\leq s\leq 2$. For $s=1$, we have $e_{1}\cdot 2t\in R$. For $s=2$, by (\ref{chernclassd}) and (\ref{chowringY}), it follows that 
\begin{equation*}
	e_{1}^{6}\cdot d_{3}=2e_{1}e_{2}e_{3}\cdot (4t^{3}-6e_{1}t^{2}+4e_{2}t-2e_{3})=8x_{0}\in R.
\end{equation*}
Thus, $\tau(\HSpin(2n))\leq 2$ for $s=1$ and $\tau(\HSpin(n))\leq 2^{3}$ for $s=2$. Since $\tau(\Spin(2n))=2^{s-1}$, the inequalities in (\ref{eq:inequalityHspintwossminus}) follow.

Now, assume that $s\geq 3$.	We define $m$ as follows:
	\begin{equation*}
	m=\begin{cases}2^{s} & \text{ if } n=3\cdot 2^{s},\\ 2^{s-2} & \text{ if } n=2^{s}.\end{cases}
\end{equation*}
	By Kummer's theorem, we have $v_{2}\big({2n\choose m}\big)=1$ when $n=3\cdot 2^{s}$, and $v_{2}\big({2n\choose m}\big)=3$ when $n=2^s$. Therefore, by (\ref{eq:twoniti}), there exists an odd integer $a$ (depending on $n$) such that 
	\begin{equation}\label{eq:trxerx}
		2a\cdot t^{2n-m}\in \bar{R} \,\,\text{ if }\, n=3\cdot 2^{s},\quad  8a\cdot t^{2n-m}\in \bar{R} \,\,\text{ if }\, n=2^s.
	\end{equation}

	Next, let $J=J_{s}$, where $J_{s}$ denotes the Totaro-decomposable subset as in the proof of Corollary \ref{cor:intervalupperbd}. Consider
	\begin{equation*}
	x=e_{1}^{\deg J}\cdot t^{2n-m}\cdot d(I\setminus \{n-m+1\}),
	\end{equation*}
	where $I=[1,n-1]\setminus J$. By Proposition \ref{prop:totaro2divisible} and by (\ref{eq:trxerx}), we have
	\begin{equation*}
		\hat{v}_{2}(2a\cdot x)=p+1 \text{ if $n=3\cdot 2^{s}$} \,\, \text{ and }\,\, \hat{v}_{2}(8a\cdot x)=p+3 \text{  if $n= 2^{s}$}, 
	\end{equation*}
	where 
	$p=\hat v_2(e_1^{\deg J} c(I) t^{n-1})=n-2s+1$. 
	Hence, the inequalities in (\ref{eq:inequalityHspintwossminus}) follow from Theorem \ref{thm:totaromain}.\end{proof}

The intervals  where $\tau(\HSpin (2n))=\tau(\Spin (2n))$ in Theorem \ref{thm:main} can be extended as follow. Namely, if $n$ is divisible by $2$ but not by $4$, then, in most cases, the torsion index can be determined exactly using the results of \cite[Sections 7, 8, and 9]{Totaro_Spin}, as shown in the following corollary.

\begin{corollary}\label{cor:3powerthreeupp}
Let $n \in (2^s + m_0 + 1, n_0) \cup (n_0, 2^{s + 1})$ for some integer $s \geq 3$. If $v_2(n) = 1$, then $\tau(\HSpin(2n)) = \tau(\Spin(2n))$.
\end{corollary}
\begin{proof}
By \cite[Lemma 7.1]{Totaro_Spin}, there exists a Totaro-decomposable subset of $[1,2^s+m_0]$ of degree $2^{2s-1}-1$.
If $2^s + m_0 + 1 < n < n_0$, we denote this Totaro-decomposable subset by $J$. Note that $n-1 > 2^s+m_0$, so $n-1 \notin J$.

If $s = 6$, then by \cite[the fisrt (unnumbered) Example in Section 8]{Totaro_Spin}, the whole set $[1, n_0-1] = [1,90]$ is Totaro-decomposable of degree $2^{2s}-1=4095$.
If $s \ne 6$, then a Totaro-decomposable subset of $[1, n_0-1]$ of degree $2^{2s}-1$ is constructed in \cite[Lemma 8.2]{Totaro_Spin}. 
(Precisely, \cite[Lemma 8.2]{Totaro_Spin} shows that such a Totaro-decomposable subset exists if certain arithmetic conditions hold, which are verified in \cite[end of Section 8]{Totaro_Spin} for $ 2 \leq s \leq 5 $ and $ 7 \leq s \leq 20 $, and in \cite[Section 9]{Totaro_Spin} for all $ s \geq 21 $.)
If $n_0 < n < 2^{s+1}$, we denote this Totaro-decomposable subset by $J$ (whether $ s = 6 $ or $ s \neq 6 $). Again, $n - 1 > n_{0}-1$, so $n-1 \notin J$.

These subsets are constructed in \cite{Totaro_Spin} through recursive procedures, making it difficult to write them explicitly right away.
However, since we are now assuming $ v_2(n) = 1 $, we can always assert that $ J $ is a strongly Totaro-decomposable subset of $ [1, n-1] $ by Remark \ref{rem:almostoddnstronglytotarodec}. Additionally, as noted earlier, we always have $ n - 1 \notin J $, so $ I := [1, n-1] \setminus J $ contains $ [n - 2^{v_2(n)} + 1, n - 1] = \{ n - 1 \} $. Thus, by Proposition \ref{prop:totaro2divisible}  
\begin{equation*}
\hat{v}_{2}(e_{1}^{\deg J}\cdot d(I \setminus \{n - 1\}) \cdot d_{2}^{n-1} )=\hat{v}_2 (e_1^{\deg J} c(I) t^{n-1}).
\end{equation*}
By Lemma \ref{lem:totarodecomp} and Theorem \ref{thm:totaromain}$, \hat{v}_2 (e_1^{\deg J} c(I) t^{n-1}) = \tau_2 (\HSpin(2n)) $, so the claim follows.\end{proof}
\begin{remark}
This argument can probably be generalized to bigger values of $v_2(n)$, provided we still have $n - 2^{v_2(n)} + 1 > 2^s + m_0$ or $n - 2^{v_2(n)} + 1 > n_0 - 1$ (we certainly need these inequalities to apply Proposition \ref{prop:totaro2divisible}). However, the challenge is that the definition of a strongly Totaro-decomposable subset becomes more difficult to verify. For instance, when $v_2(n) = 2$, we already need to check that $3 \notin J$. To verify this definition, we need a better understanding -- and potentially a modification -- of the sets recursively constructed in \cite[Sections 7, 8]{Totaro_Spin}. Therefore, such a generalization would require separate research, which the authors hope to undertake in future publications.
\end{remark}

\appendix\label{appA}
\section{Lower bounds on $\tau(\HSpin(12))$ and $\tau(\HSpin(16))$}

\begin{example}\label{ex:hspintwelve}
	We show that $\tau_{2}(\HSpin(12))\geq 2$,  and thus $\tau_{2}(\HSpin(12))=2$ by Theorem \ref{thm:main}. Since $\bar{R}$ is generated by $d_{1},\ldots, d_{5}$, and $2t$, by using the relations
	\begin{equation*}
		t^{6}=-c_{1}t^{5}-c_{2}t^{4}-c_{3}t^{3}-c_{4}t^{2}-c_{5}t
	\end{equation*}
	and
	\begin{align*}
		d_{1}&=6t+c_{1},\\
		d_{2}&=15t^{2}+5tc_{1}+c_{2},\\
		d_{3}&=20t^{3}+10t^{2}c_{1}+4tc_{2}+c_{3},\\
		d_{4}&=15t^{4}+10t^{3}c_{1}+6t^{2}c_{2}+3tc_{3}+c_{4},\\
		d_{5}&=6t^{5}+5t^{4}c_{1}+4t^{3}c_{2}+3t^{2}c_{3}+2tc_{4}+c_{5},
	\end{align*}
	from (\ref{relationtandc}) and (\ref{chernclassd}), we may assume that every element of $\bar{R}$ is a polynomial in $c_{1},\ldots, c_{5}$ and $t$, such that the exponent of $t$ is at most $5$. 
	By (\ref{eq:relationcitwo}), we can further assume that every element of $\bar{R}$ is square-free in $c_{1},\ldots, c_{5}$. Moreover, it follows from the coefficients of $t^{i}$ ($1\leq i\leq 5$) in $d_{1},\ldots, d_{5}$ that every element of $\bar{R}$ has a coefficient of $t^{5}$ in the following form:
	\begin{equation*}
		2m+g_{5}(c_{1},\ldots, c_{5}),
	\end{equation*}
	where $m$ denotes a positive integer, and $g_{5}$ denotes a polynomial in $c_{1},\ldots, c_{5}$ with no constant term. In summary, we can write an element $f$ in $\bar{R}$ as follows:
	\begin{equation}\label{eq:formoffg}
		f=(2m+g_{5} )t^{5}+\sum_{i=0}^{4} g_{i}t^{i},
	\end{equation} 
	for some polynomials $g_{i}$ in $c_{1},\ldots, c_{5}$.

	Since $R$ is a $\bar{R}$-algebra generated by $e_{1}$, every element $x$ of $R$ can be written as
	\begin{equation*}
		x=f_{0}+f_{1}e_{1}+\cdots +f_{15}e_{1}^{15},	
	\end{equation*}
	where $f_{i}\in \bar{R}$ is a polynomial of the form given in (\ref{eq:formoffg}). By  \cite[\S4]{Totaro_Spin} (see also Remark \ref{rmk:generalJ}), we have 
	\begin{equation*}
		2^{S_{2}(\deg J)}\cdot e_{1}^{\deg J}\in  R'
	\end{equation*}
	for any subset $J\subseteq [1,5]$, thus $2^{3}(x-f_{15}e_1^{15})\in R'$. If $x$ is homogeneous of top degree, then we may assume without loss of generality that $f_0, f_1, \ldots, f_{15}$ are also homogeneous. Decomposing the homogeneous $f_{15}$ as in (\ref{eq:formoffg}), we obtain
	\begin{equation*}
		2^3 f_{15} e_1^{15} = 2^3 ((2m+g_{5} )t^{5}+\sum_{i=0}^{4} g_{i}t^{i}) e_1^{15} = 2^3 \cdot 2mt^5 e_1^{15} \in R'
	\end{equation*}
	Thus, $2^3 x \in R'$. Since the degree of any element of $R'$ is divisible by $2^{5}$, the degree of $x$ is divisible by $2^{2}$, i.e., $\tau_{2}(\HSpin(12))\geq 2$.
\end{example}

\begin{lemma}\label{lem:ntwosdnisquare}
	Let $n=2^{s}$ for some integer $s\geq 2$. Then, $d_{i}\equiv 0 \mod 2\hat{R}$ for any $i\in [1,n-1]$ and
	\begin{equation*}
		d_{n-i}^{2}\equiv 0 \mod 2^{3}\hat{R}	
	\end{equation*}
	for any $i\in [1, \tfrac{n}{2}-1]$.
\end{lemma}
\begin{proof}
	Since ${2^{s}\choose i}$ is divisible by $2$ for any $i\in [1,2^{s}-1]$, the first assertion immediately follows from (\ref{chernclassd}). Let $n=2m$, i.e., $m=2^{s-1}$ and $k=n-i$. Applying (\ref{chernclassd}) again, we obtain
	\begin{align}
		d_{k}^{2}&=(a_{k}t^{k}+2\sum_{j=1}^{k}(-1)^{j}a_{k-j}t^{k-j}e_{j})^{2} \nonumber  \\
		&\equiv a_{k}^{2}t^{2k}+2^{2}\sum_{j=1}^{k}a_{k-j}^{2}t^{2k-2j}e_{j}^{2}+2^{2}a_{k}\sum_{j=1}^{k}(-1)^{j}a_{k-j}t^{2k-j}e_{j} \mod 2^{3}\hat{R},\label{eq:aligndksquare}
	\end{align}
		where $a_{k}={n\choose k}, a_{k-1}={n-1\choose k-1}, \cdots, a_{1}={i+1\choose 1}, a_{0}=1$. Since
		\begin{equation}\label{eq:tlzero}
			t^{l}\equiv 0 \mod 2\hat{R}
	\end{equation}	
			for any $l\geq n$, and $a_{k}$ is divisible by $2$, the first and third summands in (\ref{eq:aligndksquare}) vanish modulo $2^{3}\hat{R}$. Since $e_{j}^{2}\equiv 0 \mod 2\hat{R}$ for any $j\geq m$, and again by (\ref{eq:tlzero}), we obtain
	\begin{equation*}
		d_{k}^{2}\equiv 2^{2}\sum_{j=k-m+1}^{m-1}a_{k-j}^{2}t^{2k-2j}e_{j}^{2} \mod 2^{3}\hat{R}.
	\end{equation*}

To complete the proof, it suffices to check that the coefficients $a_{m-1}, a_{m-2}\ldots, a_{m-i+1}$ are even, i.e.,
	\begin{equation*}
		{m+i-1\choose i}, {m+i-2\choose i}, \cdots ,{m+1\choose i}
	\end{equation*}
	are divisible by $2$, which follows from Kummer's theorem.\end{proof}

\begin{lemma}\label{lem:dsevensixfivecon}
	Let $n=8$. Then, modulo $2^3\hat{R}$ we have
	\begin{align}\label{eq:dsevensix}
		d_{7}d_{6}&\equiv 2^{2}\big[ (e(1,6)+e(2,5))t^{6}+(e(1,7)+e(3,5))t^{5}+(e(3,6)+e(2,7))t^{4}\big]\\
		&+2^{2}\big[e(5,6)t^{2}+e(5,7)t+e(6,7) \big]\nonumber \end{align}
	and $d_{7}d_{5}\equiv 2^{2}\big[(e(1,7)+e(3,5))t^{4}+e(5,7)\big]$. Moreover, modulo $2^{4}\hat{R}$ we have
	\begin{equation}\label{eq:dsevensixfive}
		d_{7}d_{6}d_{5}\equiv 2^{3}\big[ (e(1,6,7)+e(2,5,7)+e(3,5,6))t^{4}+e(5,6,7) \big].
	\end{equation}
\end{lemma}
\begin{proof}
	As we have modulo $2^{2}\hat{R}$
	\begin{equation}\label{eq:dsevendsixmodfour}
		d_{7}\equiv 2(e_{1}t^{6}+e_{3}t^{4}+e_{5}t^{2}+e_{7}) \text{ and } d_{6}\equiv 2(e_{1}t^{5}+e_{2}t^{4}+e_{5}t+e_{6}),
	\end{equation}
	we obtain
	\begin{align*}
		d_{7}d_{6}&\equiv 2^{2}[e(2)t^{11}+e(1,2)t^{10}+e(1,3)t^{9}+e(2,3)t^{8}+2\cdot e(1,5)t^{7}]\\
		&+2^{2}[(e(1,6)+e(2,5))t^{6}+(e(1,7)+e(3,5))t^{5}+(e(3,6)+e(2,7))t^{4} ]\\
		&+2^{2}[e(5)^{2}t^{3}+e(5,6)t^{2}+e(5,7)t+e(6,7)] \mod 2^{3}\hat{R}.
	\end{align*}
	Since $t^{k}\equiv 0 \mod 2\hat{R}$ for any $k\geq 8$ and $e(5)^{2}\equiv 0 \mod 2\hat{R}$, the equation (\ref{eq:dsevensix}) follows.
	
	Similarly, since $d_{5}\equiv 2(e_{1}t^{4}+e_{5}) \mod 2^{2}\hat{R}$, we obtain
	\begin{equation*}
		d_{7}d_{5}\equiv 2^{2}(e_{1}t^{6}+e_{3}t^{4}+e_{5}t^{2}+e_{7})(e_{1}t^{4}+e_{5}) \mod 2^{3}\hat{R},
	\end{equation*}
	thus the statement follows. Let $b$ denote the right-hand side of the congruence equation in (\ref{eq:dsevensix}) without the factor $2^{2}$. Then, modulo $2\hat{R}$
	\begin{align*}
		b\cdot e_{5}&\equiv e(1,5,6)t^{6}+e(1,5,7)t^{5}+(e(3,5,6)+e(2,5,7))t^{4}+e(5,6,7),\\
		b\cdot e_{1}t^{4}&\equiv e(1,5,6)t^{6}+e(1,5,7)t^{5}+e(1,6,7)t^{4},
	\end{align*}
	thus the congruence equation (\ref{eq:dsevensixfive}) holds.
\end{proof}

\begin{corollary}\label{cor:eonefiffourtwosix}
	Let $n=8$. Then, we have
	\begin{equation*}
		e_{1}^{15}\cdot d_{7}^2d_{6}\equiv e_{1}^{14}\cdot d_{7}^{3}\equiv 0\mod 2^6\hat{R}.
	\end{equation*}
\end{corollary}
\begin{proof}
	It follows from (\ref{eq:dsevendsixmodfour}) that
	\begin{equation}\label{eq:dsevensquaremod}
		d_{7}^2\equiv 2^{2}(e_{1}^{2}t^{12}+e_{3}^{2}t^{8}+e_{5}^{2}t^{4})+2^{3}(e_{1}e_{7}t^{6}+e_{3}e_{5}t^{6}+e_{3}e_{7}t^{4}+e_{5}e_{7}t^{2}) \mod 2^{4}\hat{R}.
	\end{equation}
	
	In addition, we have
	\begin{equation}\label{eq:eonefifteenfourteen}
		e_{1}^{15}\equiv 2e([1,5]) \,\, \text{ and }\,\, e_{1}^{14}\equiv 2e(2,3,4,5)+2e(1,2,4,7) \mod 2^2\hat{R}.
	\end{equation}
	Since, for any $E\in \{e_1, e_2, e_4, e_5\}$, we have
	\begin{equation}\label{eq:eonefiveE}
		e([1,5])\cdot E\equiv 0 \mod 2\hat{R},
	\end{equation}
	$e_{5}^{2}=2(e(4,6)-e(3,7))$, $e_{3}^{2}\equiv e_{6} \mod 2\hat{R}$, $t^8 \in 2 \hat{R}$, and, by Lemma \ref{lem:ntwosdnisquare}, $d_7^2 \in 2^3 \hat R$, it follows that
	\begin{align*}
		e_{1}^{15}d_{7}^{2}&\equiv 2e([1,5])\cdot (2^{2}e_{3}^{2}t^{8}+2^{2}e_{5}^{2}t^{4}+2^{3}e_{3}e_{7}t^{4})\\
		&\equiv 2^{3}e([1,6])t^{8}+2^{4}e(1,2,4,5,6,7)t^{4}+2^{4}e(1,2,4,5,6,7)t^{4}\\
		&\equiv 2^{3}e([1,6])t^{8} \mod 2^{5}\hat{R}.
	\end{align*}
	Hence, using (\ref{eq:dsevendsixmodfour}) and (\ref{eq:eonefiveE}), 
	we obtain $e_{1}^{15}d_{7}^{2}d_{6}\equiv 0 \mod 2^{6}\hat{R}$.
		
	Since we have
	\begin{equation*}
		e(2,3,4,5)\cdot e(2)\equiv e(2,3,4,5)\cdot e(5)\equiv 0  \text{ and }  e(2,3,4,5)\cdot e_{3}^{2}\equiv e(2,3,4,5,6) \mod 2\hat{R},
	\end{equation*}
 $e(2,3,4,5)\cdot e_{5}^{2}\equiv 2e(2,4,5,6,7) \mod 2^{2}\hat{R}$, and $e(2,3,4,5)\cdot e(3,7)\equiv e(2,4,5,6,7) \mod 2\hat{R}$, it follows from (\ref{eq:dsevensquaremod}) that
	\begin{align*}
		e(2,3,4,5)\cdot d_{7}^{2}&\equiv 2^{2}t^{8}e(2,3,4,5,6)+2^{3}t^{4}e(2,4,5,6,7)\\
		&+2^{3}t^{6}e(1,2,3,4,5,7)+2^{3}t^{4}e(2,4,5,6,7)\\
		&\equiv 2^{2}t^{8}e(2,3,4,5,6)+2^{3}t^{6}e(1,2,3,4,5,7) \mod 2^{4}\hat{R}.
	\end{align*}
	Hence, by (\ref{eq:dsevendsixmodfour}) we get
	\begin{equation*}
		e(2,3,4,5)\cdot d_{7}^{2}\cdot d_{7}\equiv 2^{3}t^{14}e([1,6]))+2^{3}t^{8}e([2,7]) \mod 2^{5}\hat{R}.
	\end{equation*} 
	Since $t^{8} \equiv 2\sum_{i=1}^{7}e_{i}t^{8-i} \mod 2^{2} \hat{R}$, we have
	\begin{equation*}
		2^{3}t^{14}e([1,6])\equiv 2^{3}t^{8}e([2,7])\equiv 2^{4}e([1,7])t^{7} \mod 2^{5}\hat{R}, 
	\end{equation*}
	thus $e(2,3,4,5)\cdot d_{7}^{2}\cdot d_{7}\equiv 0\mod 2^{5}\hat{R}$.
	
	Similarly, we have
	\begin{equation*}
		e(1,2,4,7)\cdot d_{7}^{2}\equiv 2^{3}t^{6}e(1,2,3,4,5,7)+2^{2}t^{8}e(1,2,4,6,7) \mod 2^{4}\hat{R},
	\end{equation*}
	thus
	\begin{equation*}
		e(1,2,4,7)\cdot d_{7}^{2}\cdot d_{7}\equiv 2^{3}t^{12}e(1,2,3,4,6,7)+2^{3}t^{10}e(1,2,4,5,6,7) \mod 2^{5}\hat{R}.
	\end{equation*}
	Since $t^{8} \equiv 2\sum_{i=1}^{7}e_{i}t^{8-i} \mod 2^{2} \hat{R}$, we have
	\begin{equation*}
		2^{3}t^{12}e(1,2,3,4,6,7)\equiv 2^{3}t^{10}e(1,2,4,5,6,7)\equiv 2^{4}e([1,7])t^{7} \mod 2^{5}\hat{R}, 
	\end{equation*}
	thus $e(1,2,4,7)\cdot d_{7}^{2}\cdot d_{7}\equiv 0\mod 2^{5}\hat{R}$. Therefore, $e_{1}^{14}d_{7}^{3}\equiv 0 \mod 2^{6}\hat{R}$.\end{proof}

\begin{corollary}\label{cor:eonefifteendsevendfive}
	Let $n=8$. Then, modulo $2^{4}\hat{R}$ we have
	\begin{equation*}
		e_{1}^{15}\cdot d_{7}d_{5}\equiv 0\,\, \text{ and }\,\, 
		e_{1}^{15}\cdot d_{7}d_{6}\equiv 2^{3}e([1,7]).
	\end{equation*}	
	In particular, we have
	\begin{equation*}
		e_{1}^{15}\cdot d_{7}d_{5}d_{4}^{2}\equiv e_{1}^{15}\cdot d_{7}d_{6}d_{4}d_{3}\equiv 0 \mod 2^{6}\hat{R}.
	\end{equation*}
\end{corollary}
\begin{proof}
	It follows from Lemma \ref{lem:dsevensixfivecon} and (\ref{eq:eonefifteenfourteen}) that
	\begin{equation*}
		e_{1}^{15}\cdot d_{7}d_{5}\equiv 2^{3}\big[(e(1,7)+e(3,5))t^{4}+e(5,7)\big]\cdot e([1,5]) \mod 2^{4}\hat{R}.
	\end{equation*}
	Then, by (\ref{eq:eonefiveE}), we get
	\begin{equation*}
		e(1,7)e([1,5])\equiv e(3,5)e([1,5])\equiv e(5,7)e([1,5])\equiv 0 \mod 2\hat{R},
	\end{equation*}
	thus $e_{1}^{15}\cdot d_{7}d_{5}\equiv 0 \mod 2^{4}\hat{R}$. By Lemma \ref{lem:ntwosdnisquare}, we have $d_4^2 \equiv 0 \mod 2^2\hat{R}$, so multiplying by $e_1^{15}\cdot d_7 d_5$ gives $e_1^{15} d_7 d_5 d_4^2 \equiv 0 \mod 2^6 \hat{R}$.

	Similarly, by Lemma \ref{lem:dsevensixfivecon}, (\ref{eq:eonefifteenfourteen}), (\ref{eq:eonefiveE}), and noting that $e_3^2 e_6 \equiv 0 \mod 2 \hat{R}$, we get
	\begin{equation*}
		e_{1}^{15}\cdot d_{7}d_{6}\equiv 2^{3}e([1,7]) \mod 2^{4}\hat{R}.
	\end{equation*}
	As $d_{4}d_{3}=2^{4}m\cdot t^{7}+g$ for some odd integer $m$ and some homogeneous polynomial $g$ in $e_{1},\dots, e_{7}, t$ of degree $7$ such that the degree of $g$ with respect to $t$ is less than or equal to $6$, 
	we get
	\begin{equation*}
		e([1,7])\cdot d_{4}d_{3}=2^{4}m\cdot t^{7}\cdot e([1,7]),
	\end{equation*} 
	thus the statement follows.
\end{proof}

\begin{corollary}\label{cor:eonefifteendsevendfiveprime}
	Let $n=8$. Then, we have
	\begin{equation*}
		e_{1}^{14}\cdot d_{7}d_{6}d_{4}^{2}\equiv 0 \mod 2^{6}\hat{R}.
	\end{equation*}
\end{corollary}
\begin{proof}
	By Lemma \ref{lem:dsevensixfivecon}, modulo $2^3\hat{R}$ 
	we have
	\begin{align*}
		e(2,3,4,5)\cdot d_{7}d_{6}&\equiv 2^2 [e([1,6])t^{6}+e([1,5]\cup \{7\})t^{5}+e([2,7])],\\
		e(1,2,4,7)\cdot d_{7}d_{6}&\equiv 2^2 [e([1,5]\cup \{7\})t^{5}+e([1,4]\cup [6,7])t^{4}+e([1,2]\cup [4,7])t^{2}],
	\end{align*}
	thus by (\ref{eq:eonefifteenfourteen}) we get
	\begin{equation}\label{eq:donefourteendsevensix}
		e_{1}^{14}\cdot d_{7}d_{6}\equiv 2^{3}\big[e([1,6])t^{6}+e([1,4]\cup [6,7])t^{4}+e([1,2]\cup [4,7])t^{2}+e([2,7])\big] \mod 2^{4}\hat{R}.
	\end{equation}
	By a direct calculation, we get
	\begin{equation*}
		d_{4}^{2}\equiv r_{1}\cdot  t^{8}+ r_{2}\cdot e_{2}t^{6}+ r_{3}\cdot e_{2}^{2}t^{4}+ r_{4}\cdot e_{3}^{2}t^{2} \mod 2^{3}\hat{R},
	\end{equation*}
	for some integers $r_{1}, r_{2}, r_{3}, r_{4}$ divisible by $4$. Let $b$ denote the right-hand side of the equation (\ref{eq:donefourteendsevensix}) without the factor $2^{3}$. Then, as $t^{j}\equiv 0 \mod 2\hat{R}$ for any $j\geq 8$ and $e_{j}^{2}\equiv 0 \mod 2\hat{R}$ for any $j\geq 4$, each summand of $b$ multiplied by any of $t^{8}$, $e_{2}$, $e_{3}^{2}$ is congruent to $0$ modulo $2\hat{R}$. Hence, the statement follows.\end{proof}

\begin{corollary}\label{cor:neightupperbd}
	Let $n=8$. Then, we have 
	\begin{equation*}
		e_{1}^{13}\cdot d_{7}d_{6}d_{5}d_{4}\equiv 0 \mod 2^{6}\hat{R}.
	\end{equation*}
\end{corollary}
\begin{proof}
Write $d_{4}=-2\cdot 35\cdot e_{1}t^{3}+2d_{4}'$, where $d_{4}'=35t^{4}+15e_{2}t^{2}-5e_{3}t+e_{4}$. Let $b=e(1,6,7)+e(2,5,7)+e(3,5,6)$. Then, by (\ref{eq:dsevensixfive}) we have
	\begin{equation*}
		d_{7}d_{6}d_{5}d_{4}\equiv -2^{4}\cdot 35 e_{1}\cdot bt^{7}+ 2^{4}\big( bd_{4}'t^{4}+e(5,6,7)d_{4}'-e(1,5,6,7)\cdot 35t^{3}\big) \mod 2^{5}\hat{R}.
	\end{equation*}
	Since $t^{8}\equiv 0 \mod 2\hat{R}$, the summand $2^{4}\cdot 35bt^{8}$ of $2^{4}\cdot bd_{4}'t^{4}$ is congruent to $0$ modulo $2^{5}\hat{R}$. Since $\deg (e([1,7]))=28$ and each summand of $b$ has degree $14$, it follows that
	\begin{equation*}
		e_{1}^{13}e(5,6,7)=e_{1}^{13}be_{2}=e_{1}^{13}be_{3}=e_{1}^{13}be_{4}=0.
	\end{equation*}
	Moreover, since $e_1^{13} \in 2 \hat{R}$, we obtain
	\begin{equation*}
		e_1^{13} d_{7}d_{6}d_{5}d_{4}\equiv -2^{4}\cdot 35 e_{1}\cdot bt^{7} \cdot e_1^{13} \mod 2^{6}\hat{R}.
	\end{equation*}

	Regarding $e_1^{13}$, more exactly, by direct computation, 
	\begin{equation*}
		e_{1}^{13}=e_{1}^{8+4+1}\equiv 2e(1,3,4,5)+2e(1,2,4,6)+2e(2,4,7) \mod 2^{2}\hat{R}.
	\end{equation*}	
	Hence, we have
	\begin{equation*}
		e_{1}^{13}d_{7}d_{6}d_{5}d_{4}\equiv -2^{5}\cdot 35e_{1}bt^{7} \big(e(1,3,4,5)+e(1,2,4,6)+e(2,4,7)\big) \mod 2^{6}\hat{R}.
	\end{equation*}
	Since $e_{j}^{2}\equiv 0 \mod 2\hat{R}$ for any $j\geq 4$ and $e_{1}b=e(2,6,7)+e(1,2,5,7)+e(1,3,5,6)$, modulo $2\hat{R}$ we have
	\begin{align*}
		&e_{1}b\cdot e(1,3,4,5)\equiv e([1,7]),\\
		&e_{1}b\cdot e(1,2,4,6)\equiv e(1,2,5,7)\cdot e(1,2,4,6)= (2e(1,3)+e(4))e(2,4,5,6,7)\equiv 0, \\
		&e_{1}b\cdot e(2,4,7)\equiv e([1,7]),
	\end{align*}
	the statement follows.\end{proof}

\begin{corollary}\label{cor:lowerbdneight}
	Let $n=8$. Then, every element of the form $e_{1}^{i}f$ for some $0\leq i\leq 35$, where $f$ denotes a monomial
	in the variables 
	$d_{1},\ldots, d_{7}, 2t$ with $\deg (f)=35-i$, is congruent to $0$ modulo $2^{6}\hat{R}$. In particular, 
	\begin{equation*}
		\tau_{2}(\HSpin(16))=6.
	\end{equation*}
\end{corollary}
\begin{proof}
	Since $\dim Y=28$, it suffices to consider an element $e_{1}^{i}f$ for $0\leq i\leq 28$. Write $f=(2t)^{i_{0}}\prod_{k=1}^{7}d_{k}^{i_{k}}$ with $\sum_{k=0}^{7} i_{k}=35-i$. Then, by Lemma \ref{lem:ntwosdnisquare}, for each $0\leq j\leq 5$,
	\begin{equation}\label{eq:2tikdivisiblesjplus1}
		g:=(2t)^{i_{0}}\prod_{k=1}^{4}d_{k}^{i_{k}}\equiv 0 \mod 2^{j+1}\hat{R}
	\end{equation}
	if $\deg(g)=i_{0}+\sum_{k=1}^{4}k\cdot i_{k}\geq 4j+1$. Moreover, by direct computation,
	\begin{equation}\label{eq:eoneeightmod2}
		e_{1}^{8}\equiv 0 \mod 2\hat{R},\quad e_{1}^{16}\equiv 0 \mod 2^{3}\hat{R}.
	\end{equation}

	It follows from Lemma \ref{lem:ntwosdnisquare} that
	\begin{equation*}
		h:=\prod_{k=5}^{7}d_{k}^{i_{k}}\equiv
		\begin{cases}
			0 \mod 2^{5}\hat{R} & \text{ if } \sum_{k=5}^{7}i_{k}\geq 4,\\
			0 \mod 2^{4}\hat{R} & \text{ if } \sum_{k=5}^{7}i_{k}=3 \text{ and } (i_{7},i_{6},i_{5})\neq (1,1,1).
		\end{cases}
	\end{equation*}
	If $\deg(h)=\sum_{k=5}^{7}k\cdot i_{k}\geq 22$, then $\sum_{k=5}^{7}i_{k}\geq 4$. Similarly, if $\deg(h)\geq 19$, then $\sum_{k=5}^{7}i_{k}\geq 3$ and $(i_{7},i_{6},i_{5})\neq (1,1,1)$. Therefore,
	\begin{equation}\label{eq:kfivetosevendivisible}
		h\equiv
		\begin{cases}
			0 \mod 2^{5}\hat{R} & \text{ if } \deg(h)\geq 22,\\
			0 \mod 2^{4}\hat{R} & \text{ if } \deg(h)\geq 19,\\
			0 \mod 2^{3}\hat{R} & \text{ if } \deg(h)\geq 14.
		\end{cases}
	\end{equation}

	\medskip
	
	\noindent{\it Case 1.} Let $\deg(g)\in [21,35]$. Then, it immediately follows from (\ref{eq:2tikdivisiblesjplus1}) that $e_{1}^{i}f\equiv 0 \mod 2^{6}\hat{R}$.
	
	\medskip
	\smallskip
	
	\noindent{\it Case 2.} Let $\deg(g)\in [4j+1,4j+4]$ for  $2\leq j\leq 4$. Then, $i+\deg(h)=i+\sum_{k=5}^{7}k\cdot i_{k}\in [31-4j, 34-4j]$. Hence, by (\ref{eq:eoneeightmod2}) and Lemma \ref{lem:ntwosdnisquare}, a straightforward case-by-case analysis shows that 
	\begin{equation*}
	e_{1}^{i}\prod_{k=5}^{7}d_{k}^{i_{k}}\equiv 0 \mod 2^{5-j}\hat{R},
	\end{equation*}
	 thus the statement follows from (\ref{eq:2tikdivisiblesjplus1}).

	\medskip
	\smallskip
	
	\noindent{\it Case 3.} Let $\deg(g)\in [5,8]$. Then, $i+\deg(h)\in [27, 30]$. By (\ref{eq:2tikdivisiblesjplus1}), (\ref{eq:eoneeightmod2}), and Lemma \ref{lem:ntwosdnisquare}, we may assume that $i\in [0, 15]$. If $i\in [0,7]$, then $\deg(h)\geq 20$, thus the statement follows from (\ref{eq:2tikdivisiblesjplus1}) and (\ref{eq:kfivetosevendivisible}). Similarly, if $i\in [8,13]$, then $\deg(h)\geq 14$, thus the statement follows from (\ref{eq:2tikdivisiblesjplus1}), (\ref{eq:eoneeightmod2}), and (\ref{eq:kfivetosevendivisible}). If $i=14, 15$, then again by (\ref{eq:2tikdivisiblesjplus1}), (\ref{eq:eoneeightmod2}), and (\ref{eq:kfivetosevendivisible}), the statement holds except for the following cases:
	\begin{equation*}
		(i, \deg(h), \deg(g))=(14, 13, 8), (15, 12, 8), (15, 13, 7).
	\end{equation*}	
	In each case, one can easily show that the statement holds except for the following three elements:
	\begin{equation*}
		e_{1}^{14}\cdot d_{7}d_{6}d_{4}^{2},\quad e_{1}^{15}\cdot d_{7}d_{5}d_{4}^{2},\quad  e_{1}^{15}\cdot d_{7}d_{6}d_{4}d_{3},
	\end{equation*}
	which are congruent to $0$ modulo $2^{6}\hat{R}$ by Corollaries \ref{cor:eonefifteendsevendfive} and \ref{cor:eonefifteendsevendfiveprime}.
	
		\medskip
	\smallskip
	
	\noindent{\it Case 4.} Let $\deg(g)=0$. If $i\in [24, 28]$, then by (\ref{eq:eoneeightmod2}) $e_{1}^{i}\prod_{k=5}^{7}d_{k}^{i_{k}}\equiv 0 \mod 2^{6}\hat{R}$ except for $e_{1}^{28}d_{7}$, which is congruent to $0$ modulo $2^{6}\hat{R}$ by (\ref{eq:dsevendsixmodfour}). If $i\in [16, 23]$, then by (\ref{eq:eoneeightmod2}), $e_{1}^{i}\prod_{k=5}^{7}d_{k}^{i_{k}}\equiv 0 \mod 2^{6}\hat{R}$ except for the following two elements:
	\begin{equation*}
		e_{1}^{22}d_{7}d_{6},\,\, e_{1}^{23}d_{7}d_{5}.
	\end{equation*}
	By Lemma \ref{lem:dsevensixfivecon}, we see that each summand of $e_{1}^{23}d_{7}d_{6}$ and  $e_{1}^{22}d_{7}d_{5}$ modulo $2^6 \hat{R}$ has degree in $e_{1},\ldots, e_{7}$ is bigger than $28$, thus these elements are congruent to $0$ modulo $2^{6}\hat{R}$. Similarly, if $i\in [0, 7]$, then $\sum_{k=5}^{7}i_{k}\geq 5$ except for $e_1^7 d_7^4$. In both cases, the statement follows from Lemma \ref{lem:ntwosdnisquare}. 
	
	If $i \in [8, 13]$, then $\sum_{k=5}^{7}i_{k}\geq 4$, and the statement follows from Lemma \ref{lem:ntwosdnisquare} and (\ref{eq:eoneeightmod2}). If $i\in [14,15]$, then the statement holds by a similar argument, except for the following two elements:
	\begin{equation*}
		e_{1}^{14}d_{7}^{3},\,\, e_{1}^{15}d_{7}^{2}d_{6},
	\end{equation*}
	which are congruent to $0$ modulo $2^{6}\hat{R}$ by Corollary \ref{cor:eonefiffourtwosix}.

		\medskip
	\smallskip
	
	\noindent{\it Case 5.} Let $\deg(g)\in [1,4]$. Then, $i+\deg(h)\in [31, 34]$ and, by (\ref{eq:2tikdivisiblesjplus1}), 
	$g\equiv 0 \mod 2\hat{R}$.
	If $i\in [24,28]$, then $h\equiv 0 \mod 2\hat{R}$, thus the statement follows from (\ref{eq:eoneeightmod2}). If $i\in [16,23]$, then $\deg(h)\geq 8$, thus $\sum_{k=5}^{7}i_{k}\geq 2$. Hence, the statement follows from Lemma \ref{lem:ntwosdnisquare} and (\ref{eq:eoneeightmod2}).

	Now if $i\in [0,9]$, then $\deg(h)\geq 22$, thus by (\ref{eq:kfivetosevendivisible}) the statement holds. Similarly, if $i\in [10,12]$, then $\deg(h)\geq 19$, thus the statement follows from (\ref{eq:eoneeightmod2}) and (\ref{eq:kfivetosevendivisible}). Let $i\in [13,15]$. Since $d_{6}^{2}d_{5}\equiv d_{6}d_{5}^{2}\equiv 0 \mod 2^{4}\hat{R}$, it follows from (\ref{eq:kfivetosevendivisible}) and (\ref{eq:eoneeightmod2}) that the statement holds except for the following cases:
	\begin{equation*}
		(i, \deg(h), \deg(g))=(13, 18, 4), (15, 18, 2), (14, 18, 3).
	\end{equation*}	
	In each case, one can easily show that the statement holds except for the following three elements:
	\begin{equation}\label{eq:threeexceptionalelementss}
		e_{1}^{13}\cdot d_{7}d_{6}d_{5}\cdot d_{4},\quad e_{1}^{15}\cdot d_{7}d_{6}d_{5}\cdot d_{2},\quad e_{1}^{14}\cdot d_{7}d_{6}d_{5}\cdot d_{3}.
	\end{equation}
	The first element in (\ref{eq:threeexceptionalelementss}) is congruent to $0$ modulo $2^{6}\hat{R}$ by Corollary \ref{cor:neightupperbd}. For the second element in (\ref{eq:threeexceptionalelementss}), as $d_{2}\equiv 14e_{1}t+2e_{2} \mod 2^{2}\hat{R}$, by (\ref{eq:dsevensixfive}), we have modulo $2^{5}\hat{R}$
	\begin{equation}\label{eq:d7652}
		d_{7}d_{6}d_{5}d_{2}\equiv 2^{4}\big[ (e(1,6,7)+e(2,5,7)+e(3,5,6))t^{4}+e(5,6,7) \big]\cdot (7e_{1}t+e_{2}).
	\end{equation}
	Since each summand of the right-hand side of the equation (\ref{eq:d7652}) has degree in $e_{1},\ldots, e_{7}$ bigger than $14$, the summands multiplied by $e_{1}^{15}$ becomes trivial. Hence, $e_{1}^{15}\cdot d_{7}d_{6}d_{5}d_{2}$ is divisible by $2^{6}$ in $\hat{R}$. Similarly, for $e_{1}^{14}\cdot d_{7}d_{6}d_{5}d_{3}$, as $d_{3}\equiv 42e_{1}t^{2}+12e_{2}t+2e_{3} \mod 2^{2}\hat{R}$, the same arguments prove the statement.\end{proof}


\begin{thebibliography}{10}
	
\bibitem{EKM}
R.~Elman, N.~Karpenko, and A.~Merkurjev, \emph{The algebraic and geometric theory of quadratic forms}, vol. 56 of American Mathematical Society Colloquium Publications. American Mathematical Society, Providence, RI, 2008.

\bibitem{Kar}
N. A. Karpenko, \emph{On generic quadratic forms}, Pacific J. Math. \textbf{297}(2) (2018), 367--380.


\bibitem{KM}
N.~Karpenko and A.~Merkurjev, \emph{Canonical $p$-dimension of algebraic groups}, Adv. Math. \textbf{205}(2) (2006), 410--433.


\bibitem{Totaro_Splitting}
B.~Totaro, \emph{Splitting fields for {$E_8$}-torsors}, Duke Math. J. \textbf{121}(3) (2004), 425--455.	


\bibitem{Totaro_E8}
B.~Totaro, \emph{The torsion index of E8 and other groups}, Duke Math. J. \textbf{129}(2) (2005), 219--248.

\bibitem{Totaro_Spin}
B.~Totaro, \emph{The torsion index of the spin groups}, Duke Math. J. \textbf{129}(2) (2005) 249--290.	




\end{thebibliography}
\end{document}